\documentclass[a4paper,11pt]{amsart}
\usepackage{amssymb,mathtools}
\usepackage[margin=2.6cm]{geometry}
\usepackage[T1]{fontenc}
\usepackage{lmodern}
\usepackage{tikz-cd}
\allowdisplaybreaks[1]

\numberwithin{equation}{section}

\newtheorem{thm}{Theorem}[section]
\newtheorem{prop}[thm]{Proposition}
\newtheorem{lem}[thm]{Lemma}
\newtheorem{cor}[thm]{Corollary}

\theoremstyle{definition}
\newtheorem{defn}[thm]{Definition}

\theoremstyle{remark}
\newtheorem{rem}[thm]{Remark}

\newcommand*{\Coeff}{\mathbb{K}}
\newcommand*{\BigCat}[1]{%
  \expandafter\ifx\expandafter\relax\detokenize{#1}\relax
    \mathcal{C}%
  \else
    \presubscript{#1}{\mathcal{C}}%
  \fi
}
\newcommand*{\Sbimod}[1]{%
  \expandafter\ifx\expandafter\relax\detokenize{#1}\relax
    \mathcal{S}\mathrm{bimod}%
  \else
    \presubscript{#1}{\mathcal{S}\mathrm{bimod}}%
  \fi
}

\DeclareMathOperator{\Hom}{Hom}
\DeclareMathOperator{\End}{End}
\DeclareMathOperator{\supp}{supp}

\DeclareMathOperator{\grk}{grk}
\DeclareMathOperator{\Parity}{Parity}
\DeclareMathOperator{\diag}{diag}
\DeclareMathOperator{\Ker}{Ker}

\DeclareMathOperator{\Aut}{Aut}
\newcommand*{\Z}{\mathbb{Z}}
\newcommand*{\presubscript}[2]{\prescript{}{#1}{#2}}
\newcommand*{\ch}{\mathrm{ch}}
\newcommand*{\triv}{\mathrm{triv}}
\newcommand*{\C}{\mathbb{C}}

\newcommand*{\id}{\mathrm{id}}

\title{On one-sided singular Soergel bimodules}
\author{Noriyuki Abe}
\address{Graduate School of Mathematical Sciences, the University of Tokyo, 3-8-1 Komaba, Meguro-ku, Tokyo 153-8914, Japan.}
\email{abenori@ms.u-tokyo.ac.jp}
\begin{document}
\begin{abstract}
We establish a theory of one-sided singular Soergel bimodules which is a generalization of a part of Williamson's theory.
We use a formulation of Soergel bimodules developed by the author.
\end{abstract}
\maketitle
\section{Introduction}
Attached to a Coxeter system $(W,S)$, we have the Hecke algebra $\mathcal{H}$.
By a Hecke category, we mean a categorification of $\mathcal{H}$.
Such categories play an important role in representation theory.
For example, a tilting character formula for algebraic representations of a reductive group over a positive characteristic field is described in terms of such a categorification \cite{MR3805034,MR3868004}.

There are several categorifications of $\mathcal{H}$.
The first one is geometric.
Assume that $(W,S)$ is the Weyl group of a Kac-Moody group $G$ over $\C$ with the Borel subgroup $B$ and we consider sheaves of $\Coeff$-vector spaces on the flag variety $G/B$ where $\Coeff$ is a field.
If the characteristic of $\Coeff$ is zero, the category of $B$-equivariant semisimple complexes on the flag variety gives a Hecke category.
However, this category does not work well if the characteristic of $\Coeff$ is positive.
The category of parity sheaves, introduced in \cite{MR3230821}, works well even when the characteristic of $\Coeff$ is positive and it gives a good geometric Hecke category.

We also have an algebraic categorification.
It was defined by Soergel and it is now called the category of Soergel bimodules~\cite{MR2329762}.
His theory works well when the coefficient field $\Coeff$ has characteristic zero, but does not work when the characteristic of $\Coeff$ is positive for infinite Coxeter groups.
Note that Soergel bimodules depend on a choice of representation of $W$.
To be precise, Soergel's theory relies upon the assumption that the representation is reflection faithful. Over a characteristic zero field, we always have a reflection faithful representation. However, for example, when $(W,S)$ is an affine Weyl group, which is an interesting case from the point of view of modular representation theory, the natural representation arising from a reductive group is not reflection faithful over positive characteristic fields.
The author introduced a generalization of Soergel bimodules and proved that this works well even with positive characteristic fields.
In particular, it gives a Hecke category \cite{MR4321542}.

We also mention two other categorifications.
Both were defined earlier than \cite{MR4321542}.
Fiebig introduced a certain full-subcategory of the category of sheaves on a moment graph \cite{MR2395170} and proved that this is equivalent to the category of Soergel bimodules over characteristic zero fields.
He used this category to give an upper bound for primes $p$ with which Lusztig's conjecture does not hold \cite{MR2999126,MR2726602}.
The other category was introduced by Elias-Williamson \cite{MR3555156}.
They used this category to find a counterexample of Lusztig conjecture \cite{MR3671935} and also used it for a formulation of a conjecture by Riche-Williamson~\cite{MR3805034}.

Though constructed in different ways, the various Hecke categories discussed above are all equivalent.
In the characteristic zero cases, the equivalence between the category of semisimple complexes on the flag variety and the category of Soergel bimodules was proved by Soergel and it was used to prove the Koszul duality of the category $\mathcal{O}$ of complex reductive Lie algebras~\cite{MR1029692,MR1322847}.
In a positive characteristic case, under mild assumptions, it was proved that the three categories (Soergel bimodules, Fiebig's category and the category of Elias-Williamson) are equivalent to each other in \cite{MR4321542}, and the equivalence between parity sheaves and the category of Elias-Williamson is proved in \cite{MR3805034}.

The aim of this paper is to give a singular version of these stories.
We call a subset $S_{0}$ of $S$ finitary if the subgroup of $W$ generated by $S_{0}$ is finite.
Over a characteristic zero field, Williamson \cite{MR2844932} established the theory of singular Soergel bimodules based on Soergel bimodules.
Here, a singular Hecke category categorifies the module denoted by ${}^{S_{0}}\mathcal{H}^{S_{1}}$ with finitary subsets $S_0,S_1$.
Note that in this case we have ${}^{S_{0}}\mathcal{H}^{S_{1}}\simeq \triv_{S_{0}}\otimes_{\mathcal{H}_{S_{0}}}\mathcal{H}$ where $\mathcal{H}_{S_{0}}$ is the Hecke algebra attached to the parabolic subgroup for $S_{0}$ and $\triv_{S_{0}}$ is the trivial $\mathcal{H}_{S_{0}}$-module. 
The geometric construction of a singular category is the category of semi-simple complexes over a generalized flag variety equivariant for the action of a Levi subgroup.
Since the theory is based on Soergel's one, it only works with characteristic zero fields.
So we have the following natural questions.
\begin{enumerate}
\item What is an algebraic singular Hecke category in the positive characteristic case?
\item Is this algebraic singular Hecke category equivalent to the category of parity sheaves on a generalized flag variety?
\end{enumerate}
The aim of this paper is to answer these questions partially.
As we have mentioned above, Williamson constructed categories which categorifies ${}^{S_{0}}\mathcal{H}^{S_{1}}$.
In this paper, we only consider the case of $S_1 = \emptyset$.
We remark that we need some assumptions on $S_{0}$, see Section~\ref{sec:Soergel bimodules} for the details.
In particular, the action of $W_{S_{0}}$ has to be faithful.
Similar assumptions appear in some works~\cite{MR2844932,MR3644793}.

Now we are going into more details.
Fix a field $\Coeff$ and let $(V,\{\alpha_s\}_{s\in S},\{\alpha_s^\vee\}_{s\in S})$ be a realization over $\Coeff$~\cite[Definition~3.1]{MR3555156}.
We can attach $\alpha_t\in V$ (only up to $\Coeff^{\times}$) to any reflection $t\in W$.
Let $R$ be the symmetric algebra of $V$.
Let $S_0\subset S$ be a subset and denote the group generated by $S_0$ by $W_{S_0}$.
We assume that 
\begin{itemize}
\item $\#W_{S_0} < \infty$.
\item As a graded $R^{W_{S_0}}$-module, we have $R\simeq \bigoplus_{w\in W_{S_0}}R^{W_{S_0}}(-2\ell(w))$ where $\ell(w)$ is the length of $w$ and $(-2\ell(w))$ is a grading shift.
\item For any distinct reflections $t_1,t_2\in W_{S_0}$, $\{\alpha_{t_1},\alpha_{t_2}\}$ is linearly independent.
\end{itemize}
Then we define the category of singular Soergel bimodules $\Sbimod{S_0}$.
We skip the definition in the introduction.
See Section~\ref{sec:Singular Soergel bimodules} for the definition.
If $S_0 = \emptyset$, then this is the category introduced in \cite{MR4321542} and it is shown that the category $\Sbimod{\emptyset}$ is a categorification of the Hecke algebra. 
From the definition of the category, we have the following two extra structures on $\Sbimod{S_{0}}$:
\begin{itemize}
\item A right action of $\Sbimod{} = \Sbimod{\emptyset}$.
\item A grading shift $M\mapsto M(1)$.
\end{itemize}
Let $[\Sbimod{S_0}]$ be the split Grothendieck group of $\Sbimod{S_0}$.
By the above two structures, this is a right $\mathcal{H}$-module.
\begin{thm}[{Theorem~\ref{thm:classification}, Theorem~\ref{thm:Categorification}}]
\begin{enumerate}
\item There is a bijection between indecomposable objects in $\Sbimod{S_0}$ up to grading shift and $W_{S_0}\backslash W$.
\item $[\Sbimod{S_0}]\simeq \triv_{S_0}\otimes_{\mathcal{H}_{S_0}}\mathcal{H}$.
\end{enumerate}
\end{thm}

Next, we state a theorem with a geometric setting.
Let $G$ be a Kac-Moody group and $B$ a Borel subgroup of $G$.
Assume that the Weyl group of $G$ is $(W,S)$.
Then $S$ corresponds to the set of simple roots.
The subset $S_0\subset S$ defines a standard parabolic subgroup $P_{S_0}$ and the generalized flag variety $\presubscript{S_0}{X} = P_{S_0}\backslash G$.
Let $\Parity_B(\presubscript{S_0}{X})$ be the category of $B$-equivariant parity sheaves on $\presubscript{S_0}{X}$.
Let $X$ be the flag variety.
Then by the convolution product, we have a right action of $\Parity_B(X)$ on $\Parity_B(\presubscript{S_0}{X})$.
Recall that $\Sbimod{S_0}$ has a right action of $\Sbimod{}$.
\begin{thm}[Theorem~\ref{thm:equivalence}]
We have $\Sbimod{S_0}\simeq \Parity_B(\presubscript{S_0}{X})$, which is compatible with the right actions.
\end{thm}
When $S_0 = \emptyset$, this can be obtained by combining the results in \cite{MR3805034,MR4321542}.
The proof of the above theorem is different from this proof and is closer to the original proof in the characteristic zero case.

\subsection*{Acknowledgment}
The author was supported by JSPS KAKENHI Grant Number 18H01107.

\section{Singular Soergel bimodules}\label{sec:Singular Soergel bimodules}
\subsection{Notation}
In this paper, we use the following notation.
Let $(W,S)$ be a Coxeter system, $\Coeff$ a noetherian integral domain.
The unit element of $W$ is denoted by $e$ and the length function is denoted by $\ell$.
We fix a realization $(V,\{(\alpha_s,\alpha_s^\vee)\}_{s\in S})$ over $\Coeff$~\cite[Definition~3.1]{MR3555156}.
We assume that $\alpha_s\neq 0$ and $\alpha_s^\vee\colon V\to \Coeff$ is surjective.
For each $t = wsw^{-1}$ with $s\in S$ and $w\in W$, we put $\alpha_t = w(\alpha_s)$.
This depends on a choice of $(w,s)$ but $\Coeff^{\times} \alpha_t$ does not depend on $(w,s)$~\cite[Lemma~2.1]{MR4321542}.
We fix such $(w,s)$ for each $t$ to define $\alpha_t$.
Let $R$ be the symmetric algebra of $V$ over $\Coeff$ and $Q$ the field of fractions of $R$.
For a subset $S_0\subset S$, let $W_{S_0}$ be the group generated by $S_0$, $R^{W_{S_0}}$ the subalgebra of $R$ consisting of $W_{S_0}$-invariants.
Each $x\in W_{S_0}\backslash W$ has the minimal length representative.
Let $x_{-}$ be this representative.
For $f\in R^{W_{S_{0}}}$, $x^{-1}(f)$ denotes $w^{-1}(f)$ where $w$ is a representative of $x$.
This does not depend on the choice of $w$.

When $S_{1}\subset S_{0}$, for $x\in W_{S_{0}}\backslash W$, let $\prescript{S_{1}}{S_{0}}{[x]}$ be the inverse image of $x$ under the natural projection $W_{S_{1}}\backslash W\to W_{S_{0}}\backslash W$.

The Bruhat order on $W$ is denoted by $\le$.
We also define the order $\le$ on $W_{S_0}\backslash W$ by $x\le y$ if and only if $x_-\le y_-$.
We define a topology on $W_{S_0}\backslash W$ as follows: a subset $I\subset W_{S_0}\backslash W$ is open if $y\in I$, $x\in W_{S_{0}}\backslash W$, $x\ge y$ implies $x\in I$.

The algebra $R$ is a graded algebra with $\deg(V) = 2$, here by graded we always mean $\Z$-graded.
For a graded $R$-module $M = \bigoplus_{i\in\Z}M^i$, we define a graded $R$-module $M(k)$ by $M(k) = \bigoplus_{i\in\Z}M(k)^i$, $M(k)^i = M^{i + k}$.
A graded $R$-module $M$ is called graded free if it is isomorphic to $\bigoplus_{i = 1}^r R(n_i)$ for some $n_1,\dots,n_r\in \Z$.
Note that in this paper graded free means graded free of finite rank.
If $M\simeq \bigoplus_{i = 1}^rR(n_i)$ is graded free, the graded rank $\grk(M)\in \Z[v,v^{-1}]$ of $M$ is defined by $\grk(M) = \sum_{i = 1}^rv^{n_i}$ where $v$ is an indeterminate.

\subsection{A category}
Throughout this section, we fix a subset $S_0\subset S$ such that $W_{S_0}$ is finite and $V$ is faithful as a $W_{S_0}$-representation.
Let $Q^{W_{S_0}}$ be the set of $W_{S_0}$-invariants in $Q$.
\begin{lem}\label{lem:invariants and fractions}
\begin{enumerate}
\item The algebra $Q^{W_{S_0}}$ is equal to the field of fractions of $R^{W_{S_0}}$.
\item Let $S_1\subset S_0$. Then the multiplication map $R^{W_{S_1}}\otimes_{R^{W_{S_0}}}Q^{W_{S_0}}\to Q^{W_{S_1}}$ is an isomorphism.
\end{enumerate}
\end{lem}
\begin{proof}
(1)
Let $Q_1$ be the field of fractions of $R^{W_{S_0}}$ and $Q_2$ the $W_{S_0}$-invariants of $Q$.
Then we have $Q_1\subset Q_2$.
Let $f\in Q_2$ and denote $f = f_1/f_2$ where $f_1,f_2\in R$.
Then we have $f = (\prod_{w\in W_{S_0}}w(f_1))/((\prod_{w\in W_{S_0},w\ne 1}w(f_1))f_2)$.
Since $f$ is $W_{S_0}$-invariant, the denominator is also $W_{S_0}$-invariant.
Hence $f\in Q_1$.

(2)
Since the map is induced by $R^{W_{S_1}}\hookrightarrow Q^{W_{S_1}}$ with a localization to $Q^{W_{S_0}}$, this is injective.
Let $f_1/f_2\in Q^{W_{S_1}}$, where $f_1,f_2\in R^{W_{S_1}}$.
Then the element $f_1/f_2$ is the image of $(\prod_{w\in W_{S_0}/W_{S_1},w\ne 1}w(f_2))f_1)\otimes (1/(\prod_{w\in W_{S_0}/W_{S_1}}w(f_2)))$.
\end{proof}

\begin{defn}
We define the category $\BigCat{S_0}$ as follows.
An object $M$ of $\BigCat{S_0}$ is a graded $(R^{W_{S_0}},R)$-bimodule with a decomposition $M\otimes_{R}Q = \bigoplus_{x\in W_{S_0}\backslash W}M^x_{Q}$ such that $fm = mx^{-1}(f)$ for any $x\in W_{S_0}\backslash W$, $m\in M^x_{Q}$ and $f\in R^{W_{S_0}}$.
A morphism $\varphi\colon M\to N$ in $\BigCat{S_0}$ is an $(R^{W_{S_0}},R)$-bimodule homomorphism $M\to N$ of degree zero such that $\varphi(M_Q^x)\subset N_Q^x$ for any $x\in W_{S_0}\backslash W$.

We also define some related objects as follows:
\begin{itemize}
\item For $n\in\Z$ and $M\in \BigCat{S_0}$, a graded $(R^{W_{S_0}},R)$-bimodule $M(n)$ is again an object of $\BigCat{S_0}$.
\item We put $\Hom^{\bullet}_{\BigCat{S_0}}(M,N) = \bigoplus_{n\in\Z}\Hom_{\BigCat{S_0}}(M,N(n))$.

\item We put $\supp_W(M) = \{w\in W_{S_0}\backslash W\mid M_Q^w\ne 0\}$.
\item Set $\BigCat{} = \BigCat{\emptyset}$.
\end{itemize}
\end{defn}
Note that the actions of $R^W$ on both sides on $M\otimes_{R}Q$ for $M\in \BigCat{S_{0}}$ coincide with each other.
Hence, if $M\to M\otimes_{R}Q$ is injective, then $M$ is a graded $R^{W_{S_0}}\otimes_{R^W}R$-module.

\begin{rem}
Assume $M\in \BigCat{S_0}$.
Then the left action of $0\ne f\in R^{W_{S_0}}$ on each direct summand $M^x_Q$ is invertible.
Hence it is also invertible on $M\otimes_{R}Q$.
Therefore $M\otimes_{R}Q$ is also a left $Q^{W_{S_0}}$-module and we have $M\otimes_{R}Q\simeq Q^{W_{S_0}}\otimes_{R^{W_{S_0}}}M\otimes_{R}Q$.
\end{rem}
\begin{defn}
Let $w\in W_{S_{0}}\backslash W$ and we define $R_{w}\in \BigCat{S_{0}}$ as follows.
As a right $R$-module, $R_w = R$ and the left action of $f\in R^{W_{S_0}}$ is given by $fr = rw^{-1}(f)$ for $r\in R_w$.
We put $(R_w)_Q^x = 0$ if $x\ne w$ and $(R_w)_Q^w = Q$.
\end{defn}
\begin{defn}
Let $S_1\subset S_0$.
For $M\in \BigCat{S_1}$, let $\pi_{S_0,S_1,*}(M)$ be the restriction of $M$ to $(R^{W_{S_0}},R)$ with 
\[
\pi_{S_0,S_1,*}(M)^x_Q = \bigoplus_{y \in \prescript{S_{1}}{S_{0}}{[x]}}M^{y}_Q.
\]
This defines a functor $\pi_{S_0,S_1,*}\colon \BigCat{S_1}\to \BigCat{S_0}$.
\end{defn}

We also define a functor $\pi_{S_0,S_1}^*\colon \BigCat{S_0}\to \BigCat{S_1}$ as follows.
For $M\in \BigCat{S_0}$, we put $\pi_{S_0,S_1}^*(M) = R^{W_{S_1}}\otimes_{R^{W_{S_0}}}M$.
Since $M\otimes_R Q$ is a left $Q^{W_{S_0}}$-module, we have
\begin{align*}
\pi_{S_0,S_1}^*(M)\otimes_{R}Q & \simeq R^{W_{S_1}}\otimes_{R^{W_{S_0}}}M\otimes_{R}Q\\
& \simeq R^{W_{S_1}}\otimes_{R^{W_{S_0}}}Q^{W_{S_0}}\otimes_{Q^{W_{S_0}}}M\otimes_{R}Q\\
& \simeq Q^{W_{S_1}}\otimes_{Q^{W_{S_0}}}M\otimes_{R}Q\\
& \simeq \bigoplus_{x\in W_{S_0}\backslash W}Q^{W_{S_1}}\otimes_{Q^{W_{S_0}}}M_Q^x.
\end{align*}

\begin{defn}
Let $x\in W_{S_{0}}\backslash W$ and $N$ a $(Q^{W_{S_{0}}},Q)$-module such that $fn = nx^{-1}(f)$ for $f\in Q^{W_{S_{0}}}$ and $n\in N$.
For $y\in \prescript{S_{1}}{S_{0}}{[x]}$, we define a $(Q^{W_{S_{1}}},Q)$-module structure on $N$ by $fn = ny^{-1}(f)$ for $f\in Q^{W_{S_{1}}}$ and $n\in N$.
The $(Q^{W_{S_{1}}},Q)$-module with this structure is written as $\epsilon_{y}^{S_{1}}N$.
\end{defn}

To give a structure of an object in $\BigCat{S_1}$ to $\pi_{S_{0},S_{1}}^{*}(M)$, we use the following lemma.

\begin{lem}\label{lem:to define pull-back}
Let $x\in W_{S_0}\backslash W$.
The map $f\otimes m\mapsto (my^{-1}(f))_y$ gives an isomorphism $Q^{W_{S_1}}\otimes_{Q^{W_{S_0}}}M_Q^x\simeq \bigoplus_{y\in \prescript{S_{1}}{S_{0}}{[x]}}\epsilon^{S_{1}}_{y}M_Q^x$ as $(Q^{W_{S_{1}}},Q)$-bimodules.
\end{lem}
\begin{proof}
Since the $W_{S_{0}}$-representation $V$ is faithful, the map $W_{S_{0}}\to \Aut(Q)$ is injective, where $\Aut(Q)$ is the set of automorphisms of the field $Q$.
By Galois descent \cite[V.10.4, Proposition~8]{MR1994218} we have $Q\otimes_{Q^{W_{S_{0}}}}M_{x}^{Q}\simeq \bigoplus_{z\in \prescript{\emptyset}{S_{0}}{[x]}}\epsilon_{z}^{\emptyset}M_{x}^{Q}$, here the isomorphism is given by $f\otimes m\mapsto (mz^{-1}(f))_{z\in \prescript{\emptyset}{S_{0}}{[x]}}$.
This is $W_{S_{1}}$-equivariant where $w\in W_{S_{1}}$ acts as
\begin{itemize}
\item $f\otimes m\mapsto w(f)\otimes m$ on the left hand side.
\item $(m_{z})_{z\in\prescript{\emptyset}{S_{0}}{[x]}}\mapsto (m_{w^{-1}z})_{z\in \prescript{\emptyset}{S_{0}}{[x]}}$ on the right hand side.
\end{itemize}
Therefore it induces an isomorphism between $W_{S_{1}}$-invariants.
The space of $W_{S_{1}}$-invariants in $Q\otimes_{Q^{W_{S_{0}}}}M_{x}^{Q}$ is $Q^{W_{S_{1}}}\otimes_{Q^{W_{S_{0}}}}M_{x}^{Q}$. (This follows, for example, by taking a basis of $M_{x}^{Q}$ as a left $Q^{W_{S_{0}}}$-vector space.)
The space of $W_{S_{1}}$-invariants in $\bigoplus_{z\in \prescript{\emptyset}{S_{0}}{[x]}}\epsilon_{z}^{\emptyset}M_{x}^{Q}$ is the set of $(m_{z})$ such that $m_{z'} = m_{z}$ if $z'\in W_{S_{1}}z$.
Let $\overline{z}$ be the image of $z\in \prescript{\emptyset}{S_{0}}{[x]}$ in $\prescript{S_{1}}{S_{0}}{[x]}$.
Then $\bigoplus_{y\in \prescript{S_{1}}{S_{0}}{[x]}}\epsilon_{y}^{S_{1}}M_{x}^{Q}\to \left(\bigoplus_{z\in \prescript{\emptyset}{S_{1}}{[x]}}\epsilon_{z}^{\emptyset}M_{x}^{Q}\right)^{W_{S_{1}}}$ (here $(\cdot)^{W_{S_{1}}}$ denotes the space of $W_{S_{1}}$-invariants) given by $(m_{y})\mapsto (m_{\overline{z}})_{z\in \prescript{\emptyset}{S_{0}}{[x]}}$ is an isomorphism as $(Q^{W_{S_{1}}},Q)$-bimodules.
Hence we get $Q^{W_{S_{1}}}\otimes_{Q^{W_{S_{0}}}}M_{x}^{Q}\simeq \bigoplus_{y\in \prescript{S_{1}}{S_{0}}{[x]}}\epsilon^{S_{1}}_{y}M_{Q}^{x}$.
\end{proof}

\begin{defn}
For $M\in \BigCat{S_0}$, we put $\pi_{S_0,S_1}^*(M) = R^{W_{S_1}}\otimes_{R^{W_{S_0}}}M$ as above.
Let $y\in W_{S_{1}}\backslash W$ and $x$ the image of $y$ in $W_{S_0}\backslash W$, namely $y\in  \prescript{S_{1}}{S_{0}}{[x]}$.
We define $(\pi_{S_0,S_1}^*M)_Q^y$ as the inverse image of $\epsilon^{S_{1}}_{y}M_{Q}^{x}$ by the isomorphism in the previous lemma.

Let $N\in \BigCat{S_0}$ and $\varphi\colon M\to N$ be a morphism in $\BigCat{S_0}$.
Then $(\id\otimes\varphi)(Q^{W_{S_1}}\otimes_{Q^{W_{S_0}}}M_Q^x)\subset Q^{W_{S_1}}\otimes_{Q^{W_{S_0}}}N_Q^x$ for each $x\in W_{S_0}\backslash W$.
By the lemma below, $\id\otimes \varphi$ is a morphism $\pi_{S_0,S_1}^*(M)\to \pi_{S_0,S_1}^*(N)$.
Hence we get a functor $\pi_{S_0,S_1}^*\colon \BigCat{S_0}\to \BigCat{S_1}$.
\end{defn}

\begin{lem}\label{lem:automatically morphism in C}
Let $x\in W_{S_0}\backslash W$ and assume that for each $y\in \prescript{S_{0}}{S_{1}}{[x]}$.
We also assume that $fm = my^{-1}(f)$ for any $f\in Q^{W_{S_1}},m\in M_i^y$ where $i = 1,2$.
Then any $(Q^{W_{S_1}},Q)$-bimodule homomorphism $\bigoplus_y M_1^y\to \bigoplus_y M_2^y$ sends $M_1^y$ to $M_2^y$.
\end{lem}
\begin{proof}
Consider the algebra homomorphism $Q^{W_{S_{0}}}\to Q$ defined by $f\mapsto x^{-1}(f)$ and regard $Q$ as a $Q^{W_{S_{0}}}$-algebra via this map.
Then both $\bigoplus_y M_1^y$ and $\bigoplus_y M_2^y$ are $Q^{W_{S_1}}\otimes_{Q^{W_{S_0}}}Q$-bimodules and the homomorphism is a $Q^{W_{S_1}}\otimes_{Q^{W_{S_0}}}Q$-module homomorphism.
Let $z\in W_{S_1}\backslash W_{S_0}$.
We have $Q^{W_{S_1}}\otimes_{Q^{W_{S_0}}}Q\simeq \bigoplus_{y\in \prescript{S_{0}}{S_{1}}{[x]}} Q$ by Lemma~\ref{lem:to define pull-back}.
Define $e_z\in Q^{W_{S_1}}\otimes_{Q^{W_{S_0}}}Q$ such that the image of $e_z$ in $\bigoplus_y Q$ is given by $(\delta_{zy})_y$ where $\delta_{zy}$ is Kronecker's delta.
Then $\varphi(M_1^z) = \varphi(e_z(\bigoplus_y M_1^y))\subset e_z(\bigoplus_y M_2^y) = M_2^z$.
\end{proof}

\begin{rem}
Lemma~\ref{lem:automatically morphism in C} may fail if $V$ is not a faithful $W_{S_{0}}$-representation.
\end{rem}

\begin{lem}\label{lem:adjointness in singular Soergel bimodules}
The pair $(\pi_{S_0,S_1}^*,\pi_{S_0,S_1,*})$ is an adjoint pair.
\end{lem}
\begin{proof}
We have $\Hom_{(R^{W_{S_1}},R)}(R^{W_{S_1}}\otimes_{R^{W_{S_0}}}M,N)\simeq \Hom_{(R^{W_{S_0}},R)}(M,N)$ by the universality of tensor products.
We prove that this isomorphism preserves the morphisms in $\BigCat{S_0}$ and $\BigCat{S_1}$.
Assume that $\varphi\colon \pi_{S_0,S_1}^*(M)\to N$  and $\psi\colon M\to \pi_{S_0,S_1,*}(N)$ corresponds to each other by this isomorphism.
The correspondence is given by $\psi(m) = \varphi(1\otimes m)$ and $\varphi(f\otimes m) = f\psi(m)$.

Assume that $\varphi$ is a morphism in $\BigCat{S_1}$.
By the definition, for $m\in M^x_{Q}$ with $x\in W_{S_0}\backslash W$, we have $1\otimes m\in \bigoplus_{y \in \prescript{S_{1}}{S_{0}}{[x]}}(\pi_{S_0,S_1}^*M)_{Q}^{y}$.
Hence $\psi(m) = \varphi(1\otimes m)\in \bigoplus_{y \in \prescript{S_{1}}{S_{0}}{[x]}}N^y_Q = (\pi_{S_0,S_1,*}N)^x_Q$.

On the other hand, assume that $\psi$ is a morphism in $\BigCat{S_0}$.
Recall that we have $\pi_{S_0,S_1}^{*}M\otimes_{R}Q = \bigoplus_{x\in W_{S_0}\backslash W}Q^{W_{S_1}}\otimes_{Q^{W_{S_0}}}M^x_Q$ and this decomposition induces
\[
Q^{W_{S_1}}\otimes_{Q^{W_{S_0}}}M^x_Q \simeq \bigoplus_{y \in \prescript{S_{1}}{S_{0}}{[x]}}(\pi_{S_0,S_1}^{*}M)^y_Q.
\]
Therefore
\[
\varphi\left(\bigoplus_{y \in \prescript{S_{1}}{S_{0}}{[x]}}(\pi_{S_0,S_1}^{*}M)^y_Q\right)
=
Q^{W_{S_1}}\psi(M^x_Q)
\subset
(\pi_{S_0,S_1,*}N)^x_Q
=
\bigoplus_{y \in \prescript{S_{1}}{S_{0}}{[x]}}N^y_Q.
\]
By Lemma~\ref{lem:automatically morphism in C}, $\varphi$ is a morphism in $\BigCat{S_1}$.
\end{proof}
\begin{prop}
Let $S_2\subset S_1\subset S_0$.
Then $\pi_{S_0,S_1,*}\circ \pi_{S_1,S_2,*} \simeq \pi_{S_0,S_2,*}$ and $\pi_{S_1,S_2}^*\circ\pi_{S_0,S_1}^* \simeq \pi_{S_0,S_2}^*$.
\end{prop}
\begin{proof}
The first part is obvious and the second follows from the first and the previous proposition.
\end{proof}

Let $I\subset W_{S_0}\backslash W$ be a subset.
For $M\in \BigCat{S_0}$, we define $M^I$ to be the image of $M\to M\otimes_{R}Q = \bigoplus_{x\in W_{S_{0}}\backslash W}M^x_Q\to \bigoplus_{x\in I}M^x_Q$ and $M_I$ the inverse image of $\bigoplus_{x\in I}M^x_Q\subset \bigoplus_{x\in W_{S_{0}}\backslash W}M^x_Q$ in $M$.
It is easy to see that $M^I\otimes_{R}Q\simeq M_I\otimes_{R}Q\simeq \bigoplus_{x\in I}M^x_Q$.
Therefore $M_I,M^I\in \BigCat{S_0}$.
We write $M_{w}$, $M^{w}$ for $M_{\{w\}}$, $M^{\{w\}}$, respectively.
The proof of the following proposition is straightforward.
\begin{prop}\label{prop:support and push. pull}
Let $S_1\subset S_0$, $\pi\colon W_{S_1}\backslash W\to W_{S_0}\backslash W$ be the natural projection and $I\subset W_{S_0}\backslash W$.
Then we have $\pi_{S_0,S_1,*}(M)_I = \pi_{S_0,S_1,*}(M_{\pi^{-1}(I)})$, $\pi_{S_0,S_1}^*(M_I) = \pi_{S_0,S_1}^*(M)_{\pi^{-1}(I)}$, $\pi_{S_0,S_1,*}(M)^I = \pi_{S_0,S_1,*}(M^{\pi^{-1}(I)})$ and $\pi_{S_0,S_1}^*(M^I) = \pi_{S_0,S_1}^*(M)^{\pi^{-1}(I)}$.
\end{prop}

Let $N\in \BigCat{S_0}$ and $M\in \BigCat{}$.
Then we define $N\otimes M\in \BigCat{S_0}$ as follows.
As an $(R^{W_{S_0}},R)$-bimodule, $N\otimes M = N\otimes_{R}M$ and we put
\[
(N\otimes M)^w_{Q} =  \bigoplus_{x\in W}N^{wx^{-1}}_Q\otimes_{Q}M^x_Q
\]
for $w\in W_{S_0}\backslash W$.
Here, recall that $M\otimes_{R}Q$ is a $Q$-bimodule.
Hence we have an identification $(N\otimes_{R}M)\otimes_{R}Q \simeq (N\otimes_{R}Q)\otimes_{Q}(M\otimes_{R}Q)$.

\subsection{Soergel bimodules}\label{sec:Soergel bimodules}
We introduced the category of Soergel bimodules $\Sbimod{}\subset \BigCat{}$ in \cite{MR4321542}.
We recall the definition.
For each $s\in S$, set $B_s = R\otimes_{R^s}R(1)$.
This is a graded $R$-bimodule and this has the unique structure of an object in $\BigCat{}$ such that $\supp_W(B_s) = \{e,s\}$.
We also have $R_e\in \BigCat{}$.
Then $\Sbimod{}$ is the smallest full-subcategory of $\BigCat{}$ which contains $\{R_e\}\cup \{B_s\mid s\in S\}$ and closed under $\oplus,(1),(-1),\otimes$ and taking direct summands.

In this subsection, we assume the following.
\begin{itemize}
\item If $t_1,t_2$ are distinct reflections in $W_{S_0}$ then $\alpha_{t_1}$ and $\alpha_{t_2}$ are linearly  independent in a $\Coeff/\mathfrak{m}$-vector space $V/\mathfrak{m}V$ for any maximal ideal $\mathfrak{m}\subset \Coeff$ (GKM condition).
\item As an $R^{W_{S_0}}$-module, we have $R\simeq \bigoplus_{w\in W_{S_0}}R^{W_{S_0}}(-2\ell(w))$.
\item \cite[Assumption~3.2]{MR4321542} holds.
\end{itemize}
We also assume that $\Coeff$ is complete local.
\begin{rem}
If $V$ comes from a root system of a Kac-Moody algebra, then the third assumption is satisfied~\cite{arXiv:2012.09414}.
If $V|_{W_{S_0}}$ comes from a root datum whose torsion primes are invertible in $\Coeff$, the second assumption is satisfied~\cite{MR0342522}.
\end{rem}

\begin{rem}
Assume the second condition holds.
Since there is no degree zero homomorphism from $R^{W_{S_0}}$ to $R^{W_{S_0}}(-2\ell(w))$ except $w\ne 1$, the factor of the right hand side corresponding to $w = 1$ is $R^{W_{S_0}}\subset R$.
In particular, $R^{W_{S_0}}\hookrightarrow R$ splits.
\end{rem}

By \cite{MR4321542}, for each $w\in W$ there exists an indecomposable object $B(w)\in \Sbimod{}$ such that $\supp_W(B(w))\subset \{x\in W\mid x\le w\}$ and $B(w)^{w}\simeq R_w(\ell(w))$.
We have the following proposition.
This is well-known, but we give a proof for the sake of completeness.

\begin{prop}\label{prop:indecomposable for the longest}
Let $w_{S_0}$ be the longest element in $W_{S_0}$.
Put $\mathcal{Z}_{S_0} = \{(z_{w})\in \bigoplus_{w\in W_{S_0}}R\mid z_{wt}\equiv z_w\pmod{\alpha_t}\ (t\in T_{S_0})\}$ where $T_{S_0}$ is the set of reflections in $W_{S_0}$.
Then we have $B(w_{S_0})\simeq R\otimes_{R^{W_{S_0}}}R(\ell(w_{S_0}))\simeq \mathcal{Z}_{S_0}(\ell(w_{S_0}))$.
Here $\mathcal{Z}_{S_0}$ is regarded as an $R$-bimodule via $f(z_w)g = (w^{-1}(f)z_wg)$ for $(z_w)\in \mathcal{Z}$ and $f,g\in R$.
\end{prop}
In particular $R\otimes_{R^{W_{S_{0}}}}R$ is indecomposable in $\mathcal{C}$ and it is also indecomposable as an $R$-bimodule by Lemma~\ref{lem:automatically morphism in C}.

We prove this proposition in the rest of this subsection.
We put $\mathcal{Z} = \mathcal{Z}_{S_{0}}$ to simplify the notation.
We have $\mathcal{Z}\otimes_{R}Q\simeq \bigoplus_{w\in W_{S_{0}}}Q$ and from this $\mathcal{Z}$ is an object of $\BigCat{}$.
We have a map $\Phi\colon R\otimes_{R^{W_{S_{0}}}}R\to \mathcal{Z}$ defined by $f\otimes g\mapsto (w^{-1}(f)g)_{w\in W_{S_{0}}}$.

\begin{lem}\label{lem:moment graph and Rtensor R, over Q}
The map $\Phi$ is injective.
Moreover the map is bijective after tensoring with $Q$.
\end{lem}
\begin{proof}
By Lemma~\ref{lem:to define pull-back}, the map is bijective after tensoring with $Q$.
By the assumption, $R\otimes_{R^{W_{S_{0}}}}R$ is a free $R$-module, hence torsion-free.
Therefore the injectivity of $R\otimes_{R^{W_{S_{0}}}}R\to \mathcal{Z}$ follows.
\end{proof}

Hence $R\otimes_{R^{W_{S_{0}}}}R$ has a structure of an object of $\BigCat{}$.
The following lemma follows from the argument in \cite[2.2]{MR1173115}.
\begin{lem}\label{lem:Soergel's element}
For each $w\in W_{S_{0}}$, there exists $F_{w}\in R\otimes_{R^{W_{S_0}}}R$ whose image $(z_x)\in \mathcal{Z}$ satisfies
\begin{enumerate}
\item $z_{w}$ is of degree $2\ell(ww_{S_{0}})$ non-zero element.
\item if $z_{x}\ne 0$, then $x\le w$.
\end{enumerate}
\end{lem}

\begin{lem}\label{lem:succ quot of Z}
Let $I\subset W_{S_{0}}$ be a closed subset and $w\in I$ a maximal element.
Set $I' = I\setminus\{w\}$.
We have an embeddings $(R\otimes_{R^{W_{S_{0}}}}R)_I/(R\otimes_{R^{W_{S_{0}}}}R)_{I'}\hookrightarrow (\mathcal{Z})_{I}/(\mathcal{Z})_{I'}\hookrightarrow R$ where last map is  defined by $(z_x)_{x\in W_{S_{0}}}\mapsto z_w$.
Then the first map is an isomorphism and the image of the second map is $(\prod_{wt > w}\alpha_t)R$ where $t$ runs through reflections in $W_{S_{0}}$ such that $wt > w$.
\end{lem}
\begin{proof}
Let $M$ be the image of $(R\otimes_{R^{W_{S_{0}}}}R)_I/(R\otimes_{R^{W_{S_{0}}}}R)_{I'}\hookrightarrow R$.
Let $z = (z_x)\in \mathcal{Z}_I$.
Then since $w$ is maximal in $I$, for any reflection $t\in W$ such that $wt > w$, we have $wt\notin I$.
Hence $z_{wt} = 0$.
Therefore $z_w\equiv 0\pmod{\alpha_t}$.
Hence $z_{w}\in (\prod_{wt > w}\alpha_t)R$.
In particular we have $M\subset (\prod_{wt > w}\alpha_t)R$.
Take $F_{w}\in R\otimes_{R^{W_{S_{0}}}}R$ such that the image $z' = (z'_w)\in \mathcal{Z}$ of $F_{w}$ satisfies the conditions of Lemma~\ref{lem:Soergel's element}.
Hence $z'_wR\subset M$.
If $\Coeff$ is a field, then since the graded rank of $z'_w R$ and $(\prod_{wt > w}\alpha_t)R$ are equal, we conclude $M = (\prod_{wt > w}\alpha_t)R$.
In general, the embedding $M\hookrightarrow (\prod_{wt > w}\alpha_t)R$ is surjective after tensoring $\Coeff/\mathfrak{m}$ for any maximal ideal $\mathfrak{m}\subset \Coeff$.
Hence by Nakayama's lemma, $M\hookrightarrow (\prod_{wt > w}\alpha_t)R$ is surjective.
\end{proof}

\begin{lem}
The map $R\otimes_{R^{W_{S_{0}}}}R\to \mathcal{Z}$ is an isomorphism.
\end{lem}
\begin{proof}
We prove that for any closed subset $I\subset W$ we have $(R\otimes_{R^{W_{S_{0}}}}R)_I \xrightarrow{\sim} (\mathcal{Z})_I$ by induction on $\#I$.
Let $w\in I$ be a maximal element and set $I' = I\setminus\{w\}$.
Then we have the following commutative diagram with exact rows.
\[
\begin{tikzcd}
0 \arrow[r] & (R\otimes_{R^{W}}R)_{I'}\arrow[r]\arrow[d,"\sim",sloped] & (R\otimes_{R^{W_{S_{0}}}}R)_I\arrow[r]\arrow[d] & (R\otimes_{R^{W_{S_{0}}}}R)_I/(R\otimes_{R^{W_{S_{0}}}}R)_{I'}\arrow[r]\arrow[d,"\sim",sloped] & 0\\
0\arrow[r] & (\mathcal{Z})_{I'}\arrow[r] & (\mathcal{Z})_{I}\arrow[r] & (\mathcal{Z})_{I}/(\mathcal{Z})_{I'}\arrow[r] & 0.
\end{tikzcd}
\]
The left vertical map is an isomorphism by inductive hypothesis and the right vertical map is an isomorphism by the above lemma.
Hence the middle vertical map is also an isomorphism.
\end{proof}

Next we prove $R\otimes_{R^{W_{S_{0}}}}R\in \Sbimod{}$.
Define the map $p_{w}\colon R\otimes_{R^{W_{S_{0}}}}R\to R$ by $p_{w}(f\otimes g) = w^{-1}(f)g$.
This is the composition of the isomorphism $R\otimes_{R^{W_{S_{0}}}}R\simeq \mathcal{Z}$ and the projection $\mathcal{Z}\to R$ to $w$-part.
Fix a reduced expression $w_{S_{0}} = s_1\ldots s_{\ell(w_{S_{0}})}$.

To prove $R\otimes_{R^{W_{S_{0}}}}R(\ell(w_{S_{0}}))\in \Sbimod{}$, it is sufficient to prove that this is a direct summand of $B_{s_1}\otimes\cdots\otimes B_{s_{\ell(w_{S_{0}})}}$.
We construct a morphism $B_{s_1}\otimes\cdots\otimes B_{s_{\ell(w_{S_{0}})}}\to R\otimes_{R^{W_{S_{0}}}}R(\ell(w_{S_{0}}))$ which will correspond to the projection to the direct summand.
The construction will be done inductively, so for each $l$, we construct
\[
\varphi_l\colon B_{s_1}\otimes\cdots\otimes B_{s_l}\to R\otimes_{R^{W_{S_{0}}}}R(2\ell(w_{S_{0}}) - l)
\]
such that $p_{s_1\cdots s_l}(\varphi_{l}((1\otimes 1)\otimes\cdots \otimes(1\otimes 1))) = \prod_{i = l + 1}^{\ell(w_{S_{0}})}s_{l + 1}\cdots s_{i - 1}(\alpha_{s_i})$.

First we construct $\varphi_0\colon R\to R\otimes_{R^{W_{S_{0}}}}R(2\ell(w_{S_{0}}))$.
Define an element $z = (z_w)\in \mathcal{Z}$ by $z_e = \prod_{i = 1}^{\ell(w_{S_{0}})}s_1\cdots s_{i - 1}(\alpha_{s_i})$ and $z_w = 0$ for $w\ne e$.
Then $R\ni f\mapsto fz\in \mathcal{Z}(2\ell(w_{S_{0}}))$ is a morphism.
We define $\varphi_0$ as the composition of this morphism with an isomorphism $\mathcal{Z}(2\ell(w_{S_{0}}))\simeq R\otimes_{R^{W_{S_{0}}}}R(2\ell(w_{S_{0}}))$.

We assume that we have already defined $\varphi_{l - 1}$.
Set $M = B_{s_1}\otimes \cdots \otimes B_{s_{l - 1}}$, $w = s_1\ldots s_l$ and $s = s_l$.
We define $\varphi_{l}\colon M\otimes B_s\to R\otimes_{R^{W_{S_{0}}}}R(2\ell(w_{S_{0}}) - l)$ by
\begin{multline*}
M\otimes B_s\xrightarrow{\varphi_{l - 1}\otimes\id}R\otimes_{R^{W_{S_{0}}}}R\otimes_{R}B_s(2\ell(w_{S_{0}}) - l + 1) = R\otimes_{R^{W_{S_{0}}}}R\otimes_{R^s}R(2\ell(w_{S_{0}}) - l + 2)\\
\xrightarrow{f\otimes g\otimes h\mapsto f\otimes \partial_{s}(g)h}R\otimes_{R^{W_{S_{0}}}}R(2\ell(w_{S_{0}}) - l),
\end{multline*}
here $\partial_s\colon R\to R$ is defined by $\partial_s(f) = (f - s(f))/\alpha_s$.
For simplicity put $u_{l} = (1\otimes 1)\otimes\cdots \otimes(1\otimes 1)\in B_{s_1}\otimes\cdots\otimes B_{s_l}$ and we calculate $\varphi_l(u_l)$.
For any $f\otimes g\in R\otimes_{R^{W_{S_{0}}}}R$, we have
\[
p_{w}(f\otimes\partial_s(g)) = w^{-1}(f)\frac{g - s(g)}{\alpha_s} = \frac{p_{w}(f\otimes g) - s(p_{ws}(f\otimes g))}{\alpha_{s}}.
\]
Hence we have
\[
\varphi_{l}(u_l) = \frac{p_{w}(\varphi_{l - 1}(u_{l - 1})) - s(p_{ws}(\varphi_{l - 1}(u_{l - 1})))}{\alpha_{s}}.
\]
Since $\supp_{W}(M)\subset \{x\in W_{S_{0}}\mid x\le ws\}$, we have $w\notin \supp_{W}(M)$.
Therefore the $w$-part of the image of $\varphi_{l - 1}$, regarding as an element in $\mathcal{Z}$, is zero.
Hence $p_{w}(\varphi_{l - 1}(u_{l - 1})) = 0$.
On the other hand, by inductive hypothesis, we know what $p_{ws}(\varphi_{l - 1}(u_{l - 1}))$ is.
Therefore we get
\begin{align*}
\varphi_{l}(u_l)
& =
-\frac{s_l(\prod_{i = l}^{\ell(w_{S_{0}})}s_{l}\cdots s_{i - 1}(\alpha_{s_i}))}{\alpha_{s_l}}\\
& =
-\frac{s_l(\alpha_{s_l})}{\alpha_{s_l}}\prod_{i = {l + 1}}^{\ell(w_{S_{0}})}s_{l + 1}\cdots s_{i - 1}(\alpha_{s_i})
=
\prod_{i = {l + 1}}^{\ell(w_{S_{0}})}s_{l + 1}\cdots s_{i - 1}(\alpha_{s_i}).
\end{align*}
Hence the morphism $\varphi_l$ has the desired property.

Set $N = B_{s_1}\otimes\cdots\otimes B_{s_{\ell(w_{S_{0}})}}$ and $\varphi = \varphi_{\ell(w_{S_{0}})}$.
This is a morphism $N\to R\otimes_{R^{W_{S_{0}}}}R(\ell(w_{S_{0}}))$ such that $p_{w_{S_{0}}}(\varphi(u_{\ell(w_{S_{0}})})) = 1$.
Then we have an $R$-bimodule homomorphism $R\otimes R(\ell(w_{S_{0}}))\to N$ defined by $f\otimes g\mapsto fu_{\ell(w_{S_{0}})}g$.
For each $f\in R^{W_{S_{0}}}$, the left action of $f$ and the right action of $f$ on $N$ is equal.
Hence this gives $\psi\colon R\otimes_{R^{W_{S_{0}}}}R(\ell(w_{S_{0}}))\to N$.
By Lemma~\ref{lem:automatically morphism in C}, this is a morphism in $\BigCat{}$.
Consider $\varphi(\psi(1\otimes 1)) = \varphi(u_{\ell(w_{S_{0}})})\in R\otimes_{R^{W_{S_{0}}}}R(\ell(w_{S_{0}}))$.
This has the degree $-\ell(w_{S_{0}})$.
On the other hand, the degree $(-\ell(w_{S_{0}}))$-part of $R\otimes_{R^{W_{S_{0}}}}R(\ell(w_{S_{0}}))$ is $\{c(1\otimes 1)\mid c\in \Coeff\}$.
Hence we have $\varphi(u_{\ell(w_{S_{0}})}) = c(1\otimes 1)$ for some $c\in \Coeff$.
We have $c = p_{w_{S_{0}}}(c(1\otimes 1)) = p_{w_{S_{0}}}(\varphi(u_{\ell(w_{S_{0}})}))$ and this is $1$ by the construction of $\varphi$.
Hence $c = 1$ and we have $\varphi(\psi(1\otimes 1)) = 1\otimes 1$.
Since $R\otimes_{R^{W_{S_{0}}}}R(\ell(w_{S_{0}}))$ is generated by $1\otimes 1$ as an $R$-bimodule, we get $\varphi\circ\psi = \id$.
Hence $R\otimes_{R^{W_{S_{0}}}}R(\ell(w_{S_{0}}))$ is a direct summand of $N$ and therefore is an object of $\Sbimod{}$.

\begin{proof}[Proof of Proposition~\ref{prop:indecomposable for the longest}]
We prove $R\otimes_{R^{W_{S_{0}}}}R(\ell(w_{S_{0}}))\simeq B(w_{S_{0}})$.
The object $B(w_{S_{0}})$ is characterized as the unique indecomposable summand of $B_{s_1}\otimes\cdots\otimes B_{s_{\ell(w_{S_{0}})}}$ such that $B(w_{S_{0}})^{w_{S_{0}}}\simeq R(\ell(w_{S_{0}}))$.
We have $(R\otimes_{R^{W_{S_{0}}}}R(\ell(w_{S_{0}})))^{w_{S_{0}}}\simeq (\mathcal{Z}(\ell(w_{S_{0}})))^{w_{S_{0}}}\simeq R(\ell(w_{S_{0}}))$.
Therefore it is sufficient to prove that $R\otimes_{R^{W_{S_{0}}}}R$ is indecomposable.
By Lemma~\ref{lem:automatically morphism in C}, $\End_{\BigCat{}}(R\otimes_{R^{W_{S_{0}}}}R) = \End_{R\otimes R}(R\otimes_{R^{W_{S_{0}}}}R)$ is isomorphic to $R\otimes_{R^{W_{S_{0}}}}R\simeq \mathcal{Z}$.
Hence $\End_{\BigCat{}}(R\otimes_{R^{W_{S_{0}}}}R)$ has no nontrivial idempotents, and therefore $R\otimes_{R^{W_{S_{0}}}}R$ is indecomposable.
\end{proof}

\subsection{Singular Soergel bimodules}
In the rest of this section, we assume that $S_0$ satisfies the conditions in the previous section~\ref{sec:Soergel bimodules}.
We also assume that $\Coeff$ is complete local.

We remark on the following easy consequence of the assumption in the previous section~\ref{sec:Soergel bimodules}.
\begin{lem}\label{lem:pull-push is direct sum}
Let $M\in \BigCat{S_0}$.
Then $\pi_{S_0,\emptyset,*}\pi_{S_0,\emptyset}^*(M)\simeq \bigoplus_{w\in W_{S_0}}M(-2\ell(w))$.
\end{lem}

We define the category $\Sbimod{S_0}\subset \BigCat{S_0}$ of singular Soergel bimodules as follows.
\begin{defn}
An object $M\in \BigCat{S_0}$ is in $\Sbimod{S_0}$ if and only if there is $B\in \Sbimod{}$ such that $M$ is a direct summand of $\pi_{S_0,\emptyset,*}(B)$.
\end{defn}

Since any $M\in \Sbimod{S_0}$ is finitely generated as a right $R$-module, $\Hom^{\bullet}_{\Sbimod{S_0}}(M,N)\subset \Hom^{\bullet}_{\text{-$R$}}(M,N)$ is a finitely generated $R$-module.
Hence $\Hom_{\Sbimod{S_0}}(M,N)$ is a finitely generated $\Coeff$-module.
Therefore $\Sbimod{S_0}$ is Krull-Schmidt.

\begin{prop}
Let $S_1\subset S_0$.
Then the functors $\pi_{S_0,S_1,*}$ and $\pi_{S_0,S_1}^{*}$ preserve the singular Soergel bimodules.
Namely we have $\pi_{S_0,S_1,*}(\Sbimod{S_1})\subset \Sbimod{S_0}$ and $\pi_{S_0,S_1}^*(\Sbimod{S_0})\subset \Sbimod{S_1}$.
\end{prop}
\begin{proof}
The first part is obvious.
We prove the second part.
Let $M\in \Sbimod{}$ and we prove that $\pi_{S_0,S_1}^*(\pi_{S_0,\emptyset,*}(M))\in \Sbimod{S_1}$.
First we prove this when $S_1 = \emptyset$.
Then we have $\pi_{S_0,\emptyset}^*(\pi_{S_0,\emptyset,*}(M))\simeq R\otimes_{R^{W_{S_0}}}M = (R\otimes_{R^{W_{S_0}}}R)\otimes_{R}M$.
By Proposition~\ref{prop:indecomposable for the longest}, $R\otimes_{R^{W_{S_0}}}R\in \Sbimod{}$ and, by the construction of $\pi_{S_0,\emptyset}^*$, $(R\otimes_{R^{W_{S_0}}}R)\otimes_{R}M$ is isomorphic to the tensor product in the category $\BigCat{}$.
Hence $(R\otimes_{R^{W_{S_0}}}R)\otimes_{R}M\in \Sbimod{}$ since $\Sbimod{}$ is closed under the tensor product.
For general $S_1$, for $N\in \Sbimod{S_0}$, $\pi_{S_0,S_1}^*(N)$ is a direct summand of $\pi_{S_1,\emptyset,*}(\pi_{S_1,\emptyset}^*(\pi_{S_0,S_1}^*(N)))\simeq \pi_{S_1,\emptyset,*}(\pi_{S_0,\emptyset}^*(N))$ by Lemma~\ref{lem:pull-push is direct sum} and since $\pi_{S_0,\emptyset}^*(N)\in \Sbimod{}$, we have $\pi_{S_1,\emptyset,*}(\pi_{S_0,\emptyset}^*(N))\in \Sbimod{S_1}$.
Hence $\pi_{S_0,S_1}^*(N)\in \Sbimod{S_1}$.
\end{proof}

The following lemma shows a certain projectivity of Soergel bimodules.

\begin{prop}\label{prop:projectivity}
Let $I'\subset I\subset W_{S_0}\backslash W$ be closed subsets.
Then for $M,N\in \Sbimod{S_0}$, the natural map $\Hom_{\BigCat{S_{0}}}(M,N_{I})\to \Hom_{\BigCat{S_{0}}}(M,N_{I}/N_{I'})$ is surjective.
\end{prop}
\begin{proof}
We may assume $N = \pi_{S_0,\emptyset,*}(N_0)$ for $N_0\in \Sbimod{}$.
Let $\widetilde{I}$ (resp.\ $\widetilde{I}'$) be the inverse image of $I$ (resp.\ $I'$) by $W\to W_{S_0}\backslash W$.
Then these are also closed and $N_{I} = \pi_{S_0,\emptyset,*}((N_0)_{\widetilde{I}})$, $N_I/N_{I'}\simeq \pi_{S_0,\emptyset,*}((N_0)_{\widetilde{I}}/(N_0)_{\widetilde{I}'})$ by Proposition~\ref{prop:support and push. pull}.
By Lemma~\ref{lem:adjointness in singular Soergel bimodules}, we have the following commutative diagram
\[
\begin{tikzcd}
\Hom(M,N_{I})\arrow[r]\arrow[d,dash,"\sim" sloped]&  \Hom(M,N_{I}/N_{I'})\arrow[d,dash,"\sim" sloped]\\
\Hom(\pi_{S_0,\emptyset}^*(M),(N_0)_{\widetilde{I}})\arrow[r] &  \Hom(\pi_{S_0,\emptyset}^*(M),(N_0)_{\widetilde{I}}/(N_0)_{\widetilde{I}'}).
\end{tikzcd}
\]
Therefore we may assume $S_0 = \emptyset$.

By induction on $\#(I\setminus I')$, we may assume $\#(I\setminus I') = 1$.
We may assume $M = B_{s_1}\otimes\dotsb\otimes B_{s_l}$ for some $(s_1,\dots,s_l)\in S^l$.
Let $w\in I\setminus I'$.
Fix a reduced expression $w = t_1\dotsm t_r$ of $w$.
The object $N_I/N_{I'}$ is isomorphic to a direct sum of objects of a form $R_{w}(n)$ with $n\in\Z$ and since any $B_{s_1}\otimes\dotsm\otimes B_{s_l}\to R_w(n)$ factors through $B_{t_1}\otimes \dots\otimes B_{t_r}(n - \ell(w))\to B_{t_1}\otimes \dots\otimes B_{t_r}(n - \ell(w))^{w}\simeq R_w(n)$ \cite[Theorem~3.12 (3)]{MR4321542}, any homomorphism $B_{s_1}\otimes\dots\otimes B_{s_l}\to N_I/N_{I'}$ is decomposed into $B_{s_1}\otimes\dotsm\otimes B_{s_l} \to B_{t_1}\otimes\dots\otimes B_{t_r}(k)\to N_I/N_{I'}$ for some $k\in\Z$.
Hence we may assume $(s_1,\dots,s_l) = (t_1,\dots,t_r)$.
This is \cite[Corollary~ 3.18]{MR4321542}.
\end{proof}

\begin{prop}\label{prop:freeness of subquotients}
Let $I_1\subset I_2\subset W_{S_0}\backslash W$ be closed subsets.
Then for $M\in \Sbimod{S_0}$, $M_{I_2}/M_{I_1}$ is graded free as a right $R$-module.
\end{prop}
\begin{proof}
As $M$ is a direct summand of $\pi_{S_0,\emptyset,*}(N)$ for some $N\in \Sbimod{}$, we may assume $M = \pi_{S_0,\emptyset,*}(N)$.
Here note that $(M_{1}\oplus M_{2})_{I}\simeq (M_{1})_{I}\oplus (M_{2})_{I}$ for $M_{1},M_{2}\in \BigCat{S_{0}}$ and a subset $I\subset W_{S_{0}}\backslash W$.
Then, by Proposition~\ref{prop:support and push. pull}, $M_{I_i}\simeq N_{\pi^{-1}(I_i)}$ where $\pi\colon W\to W_{S_0}\backslash W$ is the natural projection and $\pi^{-1}(I_i)$ is closed for $i = 1,2$.
Therefore we may assume $S_0 = \emptyset$.

There is a finite closed subset $J\subset W$ which contains $\supp_W(M)$.
By replacing $I_i$ with $I_i\cap J$, we may assume that $I_i$ are finite.
There exists a sequence of closed subsets $I_1 = J_1 \subset J_2\subset \dotsb\subset J_r = I_2$ such that $\#(J_{i + 1}\setminus J_i) = 1$.
Since $M_{J_{i + 1}}/M_{J_i}$ is graded free as a right $R$-module~\cite[Corollary~3.18]{MR4321542}, we also have that $M_{I_2}/M_{I_1}$ is graded free.
\end{proof}

We now prove that the main structural results for one-sided singular Soergel bimodules, as proven by Williamson under more restrictive hypotheses (cf.\ \cite[Theorem~7.10, Proposition~7.11]{MR2844932}), continue to hold in our setting

\begin{thm}\label{thm:classification}
\begin{enumerate}
\item For each $w\in W_{S_0}\backslash W$, there exists an indecomposable object $\presubscript{S_0}{B(w)}\in\Sbimod{S_0}$ such that $\supp_W(\presubscript{S_0}{B(w)})\subset \{x\in W_{S_0}\backslash W\mid x\le w\}$ and $\presubscript{S_0}{B(w)}^w\simeq R_w(\ell(w_-))$.
Moreover $\presubscript{S_0}{B(w)}$ is unique up to isomorphism.
\item For any indecomposable object $B\in \Sbimod{S_0}$ there exists unique $(w,n)\in (W_{S_0}\backslash W)\times \Z$ such that $B\simeq \presubscript{S_0}{B(w)}(n)$.
\item For any $w\in W_{S_0}\backslash W$, $\pi_{S_0,\emptyset,*}B(w_-) = \presubscript{S_0}{B(w)}\oplus \bigoplus_{y < w,n\in\Z}{}\presubscript{S_0}{B(y)}(n)^{m_{y,n}}$ for some $m_{y,n}\in\Z_{\ge 0}$.
\item Let $S_1\subset S_0$, $w\in W_{S_0}\backslash W$ and $w_{1,+}\in W_{S_1}\backslash W$ the image of the maximal length representative of $w$.
Then $\pi_{S_0,S_1}^*(\presubscript{S_0}{B(w)})\simeq \presubscript{S_1}{B(w_{1,+})}(\ell(w_{S_1})-\ell(w_{S_0}))$.
\end{enumerate}
\end{thm}
\begin{proof}
Let $\pi\colon W\to W_{S_0}\backslash W$ be the natural projection.
Then since $\pi^{-1}(w)\cap \supp_W(B(w_-)) = \{w_-\}$, $\pi_{S_0,\emptyset,*}(B(w_-))^{w}\simeq B(w_-)^{\pi^{-1}(w)} = B(w_-)^{w_-}\simeq R_w(\ell(w_-))$.
Therefore there exists a unique indecomposable direct summand $\presubscript{S_0}{B(w)}\in\Sbimod{S_0}$ of $\pi_{S_0,\emptyset,*}(B(w_-))$ such that $\presubscript{S_0}{B(w)}^w\simeq R_w(\ell(w_-))$.
This object satisfies the conditions of (1).

For the uniqueness in (1) and (2), (3), it is sufficient to prove that any object $M\in \Sbimod{S_0}$ is a direct sum of $\presubscript{S_0}{B(w)}(n)$ where $w\in W_{S_0}\backslash W$ and $n\in\Z$.
Let $w\in \supp_W(M)$ be a maximal element and set $I = \{x\in W_{S_0}\backslash W\mid x\not > w\}$, $I' = I\setminus \{w\}$.
Then $I$ and $I'$ are both closed and we have $M_I = M$, $M_I/M_{I'} = M^w$.
In particular, $M^w$ is graded free by Proposition~\ref{prop:freeness of subquotients}.
Hence there exists $n\in \Z$ such that $\presubscript{S_0}{B(w)}(n)^w\simeq R_w(n + \ell(w_-))$ is a direct summand of $M^w$.
Let $i\colon \presubscript{S_0}{B(w)}(n)^w\to M^w$ and $p\colon M^w\to \presubscript{S_0}{B(w)}(n)^w$ be the embedding to and the projection from the direct summand, respectively.
By Proposition~\ref{prop:projectivity}, there exists $\widetilde{i}\colon \presubscript{S_0}{B(w)}(n)\to M$ and $\widetilde{p}\colon M\to \presubscript{S_0}{B(w)}(n)$ which induce $i$ and $p$, respectively.
Then $\widetilde{p}\circ\widetilde{i}$ induces identity on $\presubscript{S_0}{B(w)}(n)^w$.
In particular, $\widetilde{p}\circ\widetilde{i}$ is not nilpotent, hence it is an isomorphism since $\presubscript{S_0}{B(w)}(n)$ is indecomposable.
Therefore $\presubscript{S_0}{B(w)}(n)$ is a direct summand of $M$.

We prove (4).
Put $n = \ell(w_{S_1})-\ell(w_{S_0})$.
Let $\pi_0\colon W_{S_1}\backslash W\to W_{S_0}\backslash W$ be the natural projection.
Then we have $\supp_{W}(\pi_{S_0,S_1}^*(\presubscript{S_0}{B(w)})) = \pi_0^{-1}(\supp_{W}(\presubscript{S_0}{B(w)}))\subset \{x\in W_{S_1}\backslash W\mid x\le w_{1,+}\}$.
We also have $\pi_{S_0,S_1}^*(\presubscript{S_0}{B(w)})^{w_{1,+}}\simeq \presubscript{S_0}{B(w)}^{w} \simeq R_w(\ell(w_-)) \simeq \presubscript{S_1}{B(w_{1,+})^{w_{1,+}}}(\ell(w_-) - \ell((w_{1,+})_-))$.
We have $\ell(w_-) - \ell((w_{1,+})_-) = n$.
Hence $\presubscript{S_1}{B(w_{1,+})}(n)$ is a direct summand of $\pi_{S_0,S_1}^*(\presubscript{S_0}{B(w)})$.
Take $M\in \Sbimod{S_1}$ such that $\pi_{S_0,S_1}^*(\presubscript{S_0}{B(w)})\simeq \presubscript{S_1}{B(w_{1,+})}(n)\oplus M$.
We prove $M = 0$.
Let $w_+\in W$ be the maximal length representative of $w$.
We have $\pi_{S_0,\emptyset}^*(\presubscript{S_0}{B(w)})\simeq \pi_{S_1,\emptyset}^*(\presubscript{S_1}{B(w_{1,+})}(n))\oplus \pi_{S_1,\emptyset}^*(M)$ and, by the above argument, $\pi_{S_1,\emptyset}^*(\presubscript{S_1}{B(w_{1,+})}(n))$ has a direct summand $B(w_+)(-\ell(w_{S_0}))$.
Assume that we can prove $\pi_{S_0,\emptyset}^*(\presubscript{S_0}{B(w)})\simeq B(w_+)(-\ell(w_{S_0}))$.
Then we conclude $\pi_{S_1,\emptyset}^*(\presubscript{S_1}{B(w_{1,+})}(n)) \simeq B(w_+)(-\ell(w_{S_0}))$ and $\pi_{S_1,\emptyset}^*(M) = 0$.
By the definition, it implies $M = 0$.
Therefore it is sufficient to prove that $\pi_{S_0,\emptyset}^*(\presubscript{S_0}{B(w)})\simeq B(w_+)(-\ell(w_{S_0}))$, namely we may assume $S_1 = \emptyset$.

We have $\pi_{S_{0},\emptyset}^{*}(\presubscript{S_{0}}{B(w)}) = B(w_{+})(-\ell(w_{S_{0}}))\oplus M$.
Since $\pi_{S_{0},\emptyset}^{*}(\presubscript{S_{0}}{B(w)})^{\pi^{-1}(w)}\simeq \pi_{S_{0},\emptyset}^{*}(\presubscript{S_{0}}{B(w)}^{w}) = R\otimes_{R^{W}}R(\ell(w_{-}))$ is indecomposable by Proposition~\ref{prop:indecomposable for the longest}, we have $\pi_{S_{0},\emptyset}^{*}(\presubscript{S_{0}}{B(w)})^{\pi^{-1}(w)} = B(w_{+})^{\pi^{-1}(w)}(-\ell(w_{S_{0}}))$ and $M^{\pi^{-1}(w)} = 0$.
Therefore we have $w\notin \supp_W(\pi_{S_0,\emptyset,*}(M))$.
The object $\pi_{S_0,\emptyset,*}(\pi_{S_0,\emptyset}^*(\presubscript{S_0}{B(w)}))$ is a direct sum of $\presubscript{S_0}{B(w)}$ by Lemma~\ref{lem:pull-push is direct sum}.
Hence $\pi_{S_0,\emptyset,*}(M)$ is a direct sum of $\presubscript{S_0}{B(w)}$.
Since $w\notin \supp_W(\pi_{S_0,\emptyset,*}(M))$, $\pi_{S_0,\emptyset,*}(M) = 0$ and this implies $M = 0$.
\end{proof}

\subsection{Duality}
For $M\in \BigCat{S_0}$, we define $D(M) = \presubscript{S_0}{D}(M) = \Hom^{\bullet}_{\text{-$R$}}(M,R)$.
Here $\Hom^{\bullet}_{\text{-$R$}}$ is the space of right $R$-homomorphisms.
This is a graded $(R^{W_{S_0}},R)$-bimodule via $(f\varphi g)(m) = \varphi(fmg)$ for $f\in R^{W_{S_0}}$, $g\in R$, $\varphi\in D(M)$ and $m\in M$.
We have $D(M)\otimes_{R}Q \simeq\Hom_{\text{-$Q$}}(M\otimes_{R}Q,Q)\simeq \bigoplus_{w\in W_{S_0}\backslash W}\Hom_{\text{-$Q$}}(M_Q^w,Q)$.
By putting $D(M)_{Q}^{w} = \Hom_{\text{-$Q$}}(M_Q^w,Q)$, we have $D(M)\in \BigCat{S_0}$.
Since any $M\in\Sbimod{S_0}$ is free as a right $R$-module, we have $D^2(M)\simeq M$.
Note that $D(\Sbimod{})\subset \Sbimod{}$~\cite[Lemma~2.20]{MR4321542}.

\begin{prop}
Let $M\in \Sbimod{S_0}$ and $B\in \Sbimod{}$, $D(M\otimes B)\simeq D(M)\otimes D(B)$.
\end{prop}
\begin{proof}
There is a natural map $D(M)\otimes D(B)\to D(M\otimes B)$ defined by $\varphi\otimes\psi\mapsto ((m\otimes b)\mapsto \varphi(m)\psi(b))$ and since $M$ is a free right $R$-module, this is an isomorphism.
It is easy to see that this morphism is a morphism in $\BigCat{S_0}$.
\end{proof}

\begin{prop}
Let $S_1\subset S_0$.
Then $\pi_{S_0,S_1,*}(\presubscript{S_1}{D}(M))\simeq \presubscript{S_0}{D}(\pi_{S_0,S_1,*}(M))$ for $M\in \BigCat{S_1}$.
In particular, $\presubscript{S_0}{D}$ preserves $\Sbimod{S_0}$.
\end{prop}
\begin{proof}
This is obvious.
\end{proof}

\begin{lem}\label{lem:dual of M^I}
Let $M\in\BigCat{S_0}$ and $I$ a subset of $W_{S_0}\backslash W$.
Then $D(M^I)\simeq D(M)_I$.
If moreover $M$ and $D(M)^I$ are free right $R$-modules, then $D(M_I)\simeq D(M)^I$.
\end{lem}
\begin{proof}
The first one is easy.
To prove the second one, apply the first one to $D(M)$.
Then $D(D(M)^I)\simeq D(D(M))_I\simeq M_I$ since $M$ is a free right $R$-module.
If $D(M)^I$ is free, then $D(M)^I\simeq D(D(D(M)^I))\simeq D(M_I)$.
\end{proof}

\begin{prop}
Let $I_1\subset I_2\subset W_{S_0}\backslash W$ be both open or both closed.
Then for $M\in \Sbimod{S_0}$, $M_{I_2}/M_{I_1}$ is graded free as a right $R$-module.
\end{prop}
\begin{proof}
We have proved when $I_1,I_2$ are closed in Proposition~\ref{prop:freeness of subquotients}.
We assume that $I_1,I_2$ are open and set $J_i = W\setminus I_i$.
Set $N = D(M)$.
Then $D(N)_{I_i}\simeq D(N^{I_i})\simeq D(N/N_{J_i})$.
Therefore we have $D(N)_{I_2}/D(N)_{I_1}\simeq D(\Ker(N/N_{J_2}\to N/N_{J_1}))\simeq D(N_{J_1}/N_{J_2})$.
The right hand side is free since $N_{J_1}/N_{J_2}$ is free as we have proved.
\end{proof}

\begin{rem}
Let $I\subset W_{S_0}\backslash W$ be a closed or an open subset.
Then by putting $I_2 = W_{S_0}\backslash W$ and $I_1 = (W_{S_0}\backslash W)\setminus I$, we have $M_{I_2}/M_{I_1} = M/M_{I_1}\simeq M^I$ is a free right $R$-module.
\end{rem}

\subsection{Hecke algebras and Hecke modules}
We continue to assume that $S_0$ satisfies the conditions in Section~\ref{sec:Soergel bimodules} and $\Coeff$ is complete local.

Let $[\Sbimod{S_0}]$ be the split Grothendieck group of $\Sbimod{S_0}$.
Then $[M][B] = [M\otimes B]$ gives a structure of a right $[\Sbimod{}]$-module on $[\Sbimod{S_0}]$.
By $v[M] = [M(1)]$, $[\Sbimod{S_0}]$ is a $\Z[v,v^{-1}]$-module where $v$ is an indeterminate.

Let $\mathcal{H}$ be the Hecke algebra attached to $(W,S)$.
Here we use the following definition for $\mathcal{H}$:
The $\Z[v,v^{-1}]$-algebra $\mathcal{H}$ is generated by $\{H_w\mid w\in W\}$ with the following relations: $(H_s - v^{-1})(H_s + v) = 0$ for $s\in S$ and $H_{w_1w_2} = H_{w_1}H_{w_2}$ if $\ell(w_1w_2) = \ell(w_1) + \ell(w_2)$.
It is known that $\{H_w\mid w\in W\}$ is a $\Z[v,v^{-1}]$-basis of $\mathcal{H}$.
It is proved in \cite{MR4321542} that the map $\ch\colon [\Sbimod{}]\to \mathcal{H}$ defined by $\ch([B]) = \sum_{x\in W}v^{-\ell(x)}\grk(B^x)H_x$ is a $\Z[v,v^{-1}]$-algebra isomorphism.
Therefore, $[\Sbimod{S_0}]$ is a right $\mathcal{H}$-module.

It is straightforward to prove the following.
\begin{lem}
Let $S_1\subset S_0$.
\begin{enumerate}
\item We have $\pi_{S_0,S_1,*}(M\otimes B)\simeq \pi_{S_0,S_1,*}(M)\otimes B$ for $M\in \BigCat{S_1}$ and $B\in \BigCat{}$.
\item We have $\pi_{S_0,S_1}^*(M\otimes B)\simeq \pi_{S_0,S_1}^*(M)\otimes B$ for $M\in \BigCat{S_0}$ and $B\in \BigCat{}$.
\end{enumerate}
\end{lem}
Therefore the functor $\pi_{S_0,S_1,*}$ (resp.\ $\pi_{S_0,S_1}^{*}$) induces an $[\Sbimod{}]\simeq\mathcal{H}$-module homomorphism $[\Sbimod{S_0}]\to [\Sbimod{S_1}]$ (resp.\ $[\Sbimod{S_1}]\to [\Sbimod{S_0}]$).

For a subset $S_0\subset S$, let $\mathcal{H}_{S_0}$ be the $\Z[v,v^{-1}]$-subalgebra of $\mathcal{H}$ generated by $\{H_s\mid s\in S_0\}$.
This is isomorphic to the Hecke algebra attached to $(W_{S_0},S_0)$.
The trivial character $\triv\colon \mathcal{H}\to \Z[v,v^{-1}]$ is defined by $\triv(H_w) = v^{-\ell(w)}$.
The restriction of $\triv$ on $\mathcal{H}_{S_0}$ is denoted by $\triv_{S_0}$.
Then we have a right $\mathcal{H}$-module $\triv_{S_0}\otimes_{\mathcal{H}_{S_0}}\mathcal{H}$.
We define $\presubscript{S_0}{\ch}\colon \Sbimod{S_0}\to \triv_{S_0}\otimes_{\mathcal{H}_{S_0}}\mathcal{H}$ by
\[
\presubscript{S_0}{\ch}([M]) = \sum_{w\in W_{S_0}\backslash W}1\otimes v^{\ell(w_-)}\grk(M_{\ge w}/M_{>w})H_{w_-}
\]
where $M_{\ge w} = M_{\{x\in W_{S_0}\backslash W\mid x\ge w\}}$ and $M_{> w} = M_{\{x\in W_{S_0}\backslash W\mid x > w\}}$.
By the lemma below, $\presubscript{\emptyset}{\ch}$ is the same as $\ch$ defined above.
\begin{lem}
For $M\in \Sbimod{}$ and $w\in W$, we have $\grk(M^w) = v^{2\ell(w)}\grk(M_{\ge w}/M_{\ge w})$.
\end{lem}
\begin{proof}
We have $\grk(D(M)_w) = v^{-2\ell(w)}\grk(D(M)_{\not> w}/D(M)_{\not\ge w})$ by \cite[Corollary~3.18, Proposition~3.19]{MR4321542}.
Note that, as in Lemma~\ref{lem:dual of M^I}, $D(M)_w\simeq D(M^w)$ and $D(M)_{\not >w}/D(M)_{\not \ge w}\simeq D(M^{\not >w})/D(M^{\not \ge w}) = D(M/M_{>w})/D(M/M_{\ge w})\simeq D(M_{\ge w}/M_{>w})$.
Therefore, $\grk(D(M^w))  = v^{-2\ell(w)}\grk(D(M_{\ge w}/M_{>w}))$.
Hence $\grk(M^w) = v^{2\ell(w)}\grk(M_{\ge w}/M_{> w})$.
\end{proof}

\begin{lem}\label{lem:M_K for K locally closed}
Let $J_1\subset J_2, J'_1\subset J'_2$ be open subsets of $W_{S_0}\backslash W$.
If $J_2\setminus J_1 = J'_2\setminus J'_1$, then $M_{J_2}/M_{J_1}$ and $M_{J'_2}/M_{J'_1}$ are naturally isomorphic to each other.
\end{lem}
\begin{proof}
Assume that $J'_1,J'_2\subset W_{S_0}\backslash W$ are open subsets such that $J'_1\subset J_1$, $J'_2\subset J_2$ and $J_2\setminus J_1 = J'_2\setminus J'_1$.
We prove that the natural homomorphism $M_{J'_2}/M_{J'_1}\to M_{J_2}/M_{J_1}$ is an isomorphism.
We may assume that $M = \pi_{S_0,\emptyset,*}(N)$ for some $N\in\Sbimod{}$.
Let $\pi\colon W\to W_{S_0}\backslash W$ be the natural projection.
Then we have $M_{J} = N_{\pi^{-1}(J)}$ for $J = J_1,J_2,J'_1,J'_2$ by Proposition~\ref{prop:support and push. pull}.
Therefore, by replacing $M,J_1,J_2,J'_1,J'_2$ with $N,\pi^{-1}(J_1),\pi^{-1}(J_2),\pi^{-1}(J'_1),\pi^{-1}(J'_2)$, respectively, we may assume $S_0 = \emptyset$.
We prove the lemma by induction on $\#(J_2\setminus J_1)$.
If $\#(J_2\setminus J_1) = 1$, then this is \cite[Corollary~3.18]{MR4321542}.
Assume that $\#(J_2\setminus J_1) > 1$ and take an open subset $J_3$ such that $J_1\subsetneq J_3\subsetneq J_2$.
Set $J'_3 = J_3\cap J_2'$.
Then we have $J_2'\setminus J'_3 = J_2\setminus J_3, J_3'\setminus J_1' = J_3\setminus J_1$ and $J'_3\subset J_3$.
By the commutative diagram with exact rows
\[
\begin{tikzcd}
0\arrow[r] & M_{J'_3}/M_{J'_1}\arrow[r]\arrow[d,"\sim",sloped] & M_{J'_2}/M_{J'_1} \arrow[r]\arrow[d] & M_{J'_2}/M_{J'_3}\arrow[r]\arrow[d,"\sim" sloped] & 0\\
0\arrow[r] & M_{J_3}/M_{J_1}\arrow[r] & M_{J_2}/M_{J_1} \arrow[r] & M_{J_2}/M_{J_3}\arrow[r] & 0,
\end{tikzcd}
\]
we get the lemma.
\end{proof}

\begin{prop}\label{prop:push in Grothendieck group}
Let $S_1\subset S_0$.
We define $p\colon \triv_{S_1}\otimes_{\mathcal{H}_{S_1}}\mathcal{H}\to \triv_{S_0}\otimes_{\mathcal{H}_{S_0}}\mathcal{H}$ by $p(1\otimes h) = 1\otimes h$.
Then we have $\presubscript{S_0}{\ch}([\pi_{S_0,S_1,*}(M)]) = p(\presubscript{S_1}{\ch}([M]))$ for $M\in \Sbimod{S_1}$.
\end{prop}
\begin{proof}
Fix $w\in W_{S_0}\backslash W$.
Denote the image of $w_-$ in $W_{S_1}\backslash W$ by $\overline{w_-}$.
For $w\in W_{S_{0}}\backslash W$, we put $I_{\ge w} = \{x\in W_{S_{0}}\backslash W\mid x\ge w\}$ and $I_{> w} = \{x\in W_{S_{0}}\backslash W\mid x > w\}$.
Let $\pi\colon W_{S_1}\backslash W\to W_{S_0}\backslash W$ be the natural projection.
Then we have 
\[
\text{$(\pi_{S_0,S_1,*}(M))_{\ge w} = M_{\pi^{-1}(I_{\ge w})}$ and $(\pi_{S_0,S_1,*}(M))_{> w} = M_{\pi^{-1}(I_{>w})}$}.
\]
We have
\[
\text{$\pi^{-1}(I_{\ge w}) = \{x\in W_{S_1}\backslash W\mid x\ge \overline{w_-}\}$ and $\pi^{-1}(I_{>w}) = \{x\in W_{S_1}\backslash W\mid x\ge \overline{w_-}\}\setminus (W_{S_1}\backslash W_{S_0}w_-)$}.
\]
Take a sequence of open subsets $\{x\in W_{S_1}\backslash W\mid x\ge \overline{w_-}\}\setminus (W_{S_1}\backslash W_{S_0}w_-) = J_0 \subset J_1\subset \dotsb\subset J_r =  \{x\in W_{S_1}\backslash W\mid x\ge \overline{w_-}\}$ such that $\#(J_i\setminus J_{i - 1}) = 1$.
Pick $x_i\in J_i\setminus J_{i - 1}$.
Then by Lemma~\ref{lem:M_K for K locally closed}, we have $M_{J_i}/M_{J_{i - 1}}\simeq M_{\ge x_i}/M_{>x_i}$.
Therefore 
\begin{align*}
\grk((\pi_{S_0,S_1,*}M)_{\ge w}/(\pi_{S_0,S_1,*}M)_{> w}) & = \grk(M_{\{x\in W_{S_1}\backslash W\mid x\ge \overline{w_-}\}}/M_{ \{x\in W_{S_1}\backslash W\mid x\ge \overline{w_-}\}\setminus (W_{S_{1}}\backslash W_{S_0}w_{-})})\\
& = \sum_{i = 1}^r\grk(M_{\ge x_i}/M_{> x_i})\\
& =  \sum_{x\in W_{S_1}\backslash W_{S_0}}\grk(M_{\ge xw_-}/M_{> xw_-}).
\end{align*}
Hence
\begin{align*}
\presubscript{S_0}{\ch}([\pi_{S_0,S_1,*}M]) & = \sum_{w\in W_{S_0}\backslash W}1\otimes v^{\ell(w_-)}\grk((\pi_{S_0,S_1,*}(M))_{\ge w}/(\pi_{S_0,S_1,*}(M))_{> w})H_{w_-}\\
& = \sum_{w\in W_{S_0}\backslash W}\sum_{x\in W_{S_1}\backslash W_{S_0}}1\otimes v^{\ell(w_-)}\grk(M_{\ge xw_-}/M_{>xw_-})H_{w_-}.
\end{align*}
For $w\in W_{S_0}\backslash W$ and $x\in W_{S_1}\backslash W_{S_0}$, we have $(xw_-)_- = x_-w_-$.
Moreover, since $\ell(x_-w_-) = \ell(x_-) + \ell(w_-)$, we have 
\begin{align*}
1\otimes v^{\ell((xw_-)_-)}H_{(xw_-)_-} & = 1\otimes v^{\ell(x_-)}v^{\ell(w_-)}H_{x_-}H_{w_-} \\& = \triv(H_{x_-})\otimes v^{\ell(x_-)}v^{\ell(w_-)}H_{w_-} \\ &= 1\otimes v^{\ell(w_-)}H_{w_-}.
\end{align*}
Therefore
\begin{align*}
\presubscript{S_0}{\ch}([\pi_{S_0,S_1,*}M]) 
& = \sum_{w\in W_{S_0}\backslash W}\sum_{x\in W_{S_1}\backslash W_{S_0}}1\otimes v^{\ell(x_-w_-)}\grk(M_{\ge xw_-}/M_{>xw_-})H_{x_-w_-}\\
& = \sum_{y\in W_{S_1}\backslash W}1\otimes v^{\ell(y_-)}\grk(M_{\ge y}/M_{> y})H_{y_-}\\
& = p(\presubscript{S_1}{\ch}([M])).
\end{align*}
We get the proposition.
\end{proof}

\begin{thm}\label{thm:Categorification}
The map $\presubscript{S_0}{\ch}$ is an $\mathcal{H}$-module isomorphism.
\end{thm}
\begin{proof}
The split Grothendieck group $[\Sbimod{S_0}]$ has a $\Z[v,v^{-1}]$-basis $[\presubscript{S_0}{B(w)}]$ with $w\in W_{S_0}\backslash W$.
We have $\presubscript{S_0}{\ch}(\presubscript{S_0}{B(w)})\in 1\otimes H_{w_-} + \sum_{x < w}\Z[v,v^{-1}]\otimes H_{x_-}$.
Since $\{1\otimes H_{w_-}\mid w\in W_{S_0}\backslash W\}$ is a $\Z[v,v^{-1}]$-basis of $\triv_{S_0}\otimes_{\mathcal{H}_{S_0}}\mathcal{H}$, $\{\presubscript{S_0}{\ch}(\presubscript{S_0}{B(w)})\mid w\in W_{S_0}\backslash W\}$ is also a $\Z[v,v^{-1}]$-basis.
Hence $\presubscript{S_0}{\ch}$ is a $\Z[v,v^{-1}]$-module isomorphism.

Define $p$ as in Proposition~\ref{prop:push in Grothendieck group}.
Obviously $p$ is surjective and an $\mathcal{H}$-module homomorphism.
Therefore, with the previous lemma and the fact that $\ch$ and $\presubscript{S_0}{\ch}$ are bijectives, $\{\pi_{S_0,\emptyset,*}(N)\mid N\in \Sbimod{}\}$ generates $[\Sbimod{S_0}]$ as a $\Z[v,v^{-1}]$-module.
Therefore, to prove $\presubscript{S_0}{\ch}([M\otimes B]) = \presubscript{S_0}{\ch}([M])\ch([B])$, we may assume $M = \pi_{S_0,\emptyset,*}(N)$.
Note that $\ch$ is an algebra homomorphism.
Hence we have 
\begin{align*}
\presubscript{S_0}{\ch}([\pi_{S_0,\emptyset,*}(N)\otimes B]) 
& = \presubscript{S_0}{\ch}([\pi_{S_0,\emptyset,*}(N\otimes B)])\\
& = p(\ch([N\otimes B]))\\
& = p(\ch([N])\ch([B]))\\
& = p(\ch([N]))\ch([B])\\
& = \presubscript{S_0}{\ch}([\pi_{S_0,\emptyset,*}(N)])\ch([B]).
\end{align*}
We get the theorem.
\end{proof}

Recall that for each $w\in W$, we have an indecomposable Soergel bimodule $B(w)\in \Sbimod{}$.
We put $\underline{H}_{w} = \ch([B(w)])$.
By Proposition~\ref{prop:indecomposable for the longest}, we have $\ch(R\otimes_{R^{W_{S_{0}}}}R) = v^{-\ell(w_{S_{0}})}\underline{H}_{w_{S_{0}}}$ and by Lemma~\ref{lem:succ quot of Z}, we have $\underline{H}_{w_{S_{0}}} = v^{\ell(w_{S_{0}})}\sum_{w\in W_{S_0}}v^{-\ell(w)}H_w\in \mathcal{H}$.
The map $\triv_{S_0}\otimes_{\mathcal{H}_{S_0}}\mathcal{H}\to \mathcal{H}$ defined by $1\otimes h\mapsto v^{-\ell(w_{S_{0}})}\underline{H}_{w_{S_{0}}}h$ is well-defined.

\begin{lem}
Define $i\colon \triv_{S_0}\otimes_{\mathcal{H}_{S_0}}\mathcal{H}\to \mathcal{H}$ by $i(1\otimes h) = v^{-\ell(w_{S_{0}})}\underline{H}_{w_{S_{0}}}h$.
Then we have $\ch([\pi_{S_0,\emptyset}^*(M)]) = i(\presubscript{S_0}{\ch}([M]))$.
\end{lem}
\begin{proof}
By Lemma~\ref{prop:push in Grothendieck group} and Theorem~\ref{thm:Categorification}, the map $[\Sbimod{}]\to [\Sbimod{S_0}]$ induced by $\pi_{S_0,\emptyset,*}$ is surjective.
Hence we may assume $M = \pi_{S_0,\emptyset,*}(N)$ for some $N\in \Sbimod{}$.
Since $\pi_{S_0,\emptyset}^*\pi_{S_0,\emptyset,*}(N)\simeq (R\otimes_{R^{W_{S_0}}}R)\otimes_{R}N$, we have $[\pi_{S_0,\emptyset}^*\pi_{S_0,\emptyset,*}(N)] = [R\otimes_{R^{W_{S_0}}}R][N]$.
By Lemma~\ref{prop:push in Grothendieck group}, $1\otimes \ch([N]) = \presubscript{S_0}{\ch}([M])$.
Hence $\ch([\pi_{S_0,\emptyset}^*(M)]) = \ch([R\otimes_{R^{W_{S_0}}}R])\ch([N]) = i(1\otimes \ch([N])) = i(\presubscript{S_0}{\ch}([M]))$.
\end{proof}

We define some notation.
\begin{itemize}
\item Let $h\mapsto \overline{h}$ be the $\Z$-algebra involution on $\mathcal{H}$ defined by $\overline{\sum_x a_xH_x} = \sum_x \overline{a_x}H^{-1}_{x^{-1}}$, here for $f(v)\in \Z[v,v^{-1}]$, we put $\overline{f(v)} = f(v^{-1})$.
\item The map $a\otimes h\mapsto \overline{a}\otimes \overline{h}$ is well-defined on $\triv_{S_0}\otimes_{\mathcal{H}_{S_0}}\mathcal{H}$.
We write also $m\mapsto \overline{m}$ for this map.
\item Let $\omega\colon \mathcal{H}\to \mathcal{H}$ be a $\Z$-algebra anti-involution defined by $\omega(\sum_{x}a_xH_x) = \sum_{x}\overline{a_{x}}H_{x}^{-1}$.
\item For $m,m'\in \triv_{S_0}\otimes_{\mathcal{H}_{S_0}}\mathcal{H}$, take $a_{x},b_{x}\in \Z[v,v^{-1}]$ for each $x\in W_{S_0}\backslash W$ such that $m = \sum_{x}a_{x}\otimes H_{x_-}$ and $\overline{m'} = \sum_{x}b_{x}\otimes H_{x_-}$.
Then we define $\langle m,m'\rangle_{\mathcal{H},S_0} = \sum_x \overline{a_xb_x}$.
\end{itemize}

It is straightforward to see that $\langle mh,m'\rangle_{\mathcal{H},S_0} = \langle m,m'\omega(h)\rangle_{\mathcal{H},S_0}$ for $m,m'\in \triv_{S_0}\otimes_{\mathcal{H}_{S_0}}\mathcal{H}$, $h\in \mathcal{H}$.
When $S_0 = \emptyset$, we also have $\langle hm,m'\rangle_{\mathcal{H},\emptyset} = \langle m,\omega(h)m'\rangle_{\mathcal{H},\emptyset}$.

\begin{thm}[{cf.~\cite[Theorem~7.9]{MR2844932}}]\label{thm:hom formula}
For $M,N\in \Sbimod{S_0}$, the right $R$-module $\Hom^{\bullet}_{\Sbimod{S_0}}(M,N)$ is graded free and the graded rank is given by
\[
\grk\Hom^{\bullet}_{\Sbimod{S_0}}(M,N) = \langle \presubscript{S_0}{\ch}(M),\presubscript{S_0}{\ch}(N)\rangle_{\mathcal{H},S_0}.
\]
\end{thm}
\begin{proof}
If $S_0 = \emptyset$, then this is \cite[Theorem~4.6]{MR4321542}.

We prove the general case.
By Lemma~\ref{prop:push in Grothendieck group} and Theorem~\ref{thm:Categorification}, the map $[\Sbimod{}]\to [\Sbimod{S_0}]$ induced by $\pi_{S_0,\emptyset,*}$ is surjective.
Therefore we may assume $M = \pi_{S_0,\emptyset,*}(M_0)$ and $N = \pi_{S_0,\emptyset,*}(N_0)$ for some $M_0,N_0\in \Sbimod{}$.
Then 
\[
\Hom^{\bullet}_{\Sbimod{S_0}}(M,N)\simeq \Hom^{\bullet}_{\Sbimod{}}(\pi_{S_0,\emptyset}^*(\pi_{S_0,\emptyset,*}(M_0)),N_0)
\]
is graded free and the graded rank is $\langle \ch(\pi_{S_0,\emptyset}^*(\pi_{S_0,\emptyset,*}(M_0))),\ch(N_0)\rangle_{\mathcal{H},\emptyset}$.
Take $a_x,b_x\in \Z[v,v^{-1}]$ such that $\ch(M_0) = \sum_x a_xH_x$ and $\overline{\ch(N_0)} = \sum_x b_xH_x$.
For each $w\in W_{S_0}\backslash W$, set $a'_w = \sum_{y\in W_{S_0}}v^{-\ell(y)}a_{yw_-}$ and $b'_w = \sum_{y\in W_{S_0}}v^{-\ell(y)}b_{yw_-}$.
Then we have $\presubscript{S_0}{\ch}([M]) = p(\ch([M_0])) = \sum_{w\in W_{S_0}\backslash W}a'_w\otimes H_{w_-}$.
Similarly, we also have $\overline{\presubscript{S_0}{\ch}([M])} = \overline{p(\ch([M_0]))} = p(\overline{\ch([M_0])}) = \sum_{w\in W_{S_0}\backslash W}b'_w\otimes H_{w_-}$.
Hence $\langle \presubscript{S_0}{\ch}([M]),\presubscript{S_0}{\ch}([N])\rangle_{\mathcal{H},S_0} = \sum_{w\in W_{S_0}\backslash W}\overline{a'_{w}b'_{w}}$.

We have $\underline{H}_{w_{S_{0}}}H_x = v^{-\ell(x)}\underline{H}_{w_{S_{0}}}$ for any $x\in W_{S_0}$.
Hence for $x = yw_-$ with $y\in W_{S_0}$ and $w\in W_{S_0}\backslash W$, we have $\underline{H}_{w_{S_{0}}}H_x = v^{-\ell(y)}\underline{H}_{w_{S_{0}}}H_{w_-}$.
Therefore 
\begin{align*}
\underline{H}_{w_{S_{0}}}\ch([M_0]) & = \sum_{y\in W_{S_0},w\in W_{S_0}\backslash W}a_{yw_-}v^{-\ell(y)}\underline{H}_{w_{S_{0}}}H_{w_-} \\ & = \sum_{w\in W_{S_0}\backslash W}a'_w\underline{H}_{w_{S_{0}}}H_{w_-} \\ & = \sum_{w\in W_{S_0}\backslash W,y\in W_{S_0}}a'_wv^{\ell(w_{S_0}) - \ell(y)}H_{yw_-}.
\end{align*}
In the last part we use $\underline{H}_{w_{S_{0}}} = v^{\ell(w_{S_{0}})}\sum_{w\in W_{S_0}}v^{-\ell(w)}H_w$.
Since $\overline{\underline{H}_{w_{S_{0}}}} = \underline{H}_{w_{S_{0}}}$, we have $\overline{\underline{H}_{w_{S_{0}}}\ch([N_0])} = \underline{H}_{w_{S_{0}}}\overline{\ch([N_0])} = \sum_{w\in W_{S_0}\backslash W}b'_w\underline{H}_{w_{S_{0}}}H_{w_-}$.
Therefore, by using $\omega(\underline{H}_{w_{S_{0}}}) = \underline{H}_{w_{S_{0}}}$, 
\begin{align*}
\grk\Hom^{\bullet}_{\Sbimod{S_0}}(M,N) & = 
\langle \ch([\pi_{S_0,\emptyset}^*(\pi_{S_0,\emptyset,*}(M_0))]),\ch([N_0])\rangle_{\mathcal{H},\emptyset}\\
& = \langle \ch([R\otimes_{R^W}R])\ch([M_0]),\ch([N_0])\rangle_{\mathcal{H},\emptyset}\\
& = \langle v^{- \ell(w_{S_0})}\underline{H}_{w_{S_{0}}}\ch([M_0]),\ch([N_0])\rangle_{\mathcal{H},\emptyset}\\
& = \langle v^{- \ell(w_{S_0})}\ch([M_0]),\omega(\underline{H}_{w_{S_{0}}})\ch([N_0])\rangle_{\mathcal{H},\emptyset}\\
& = \langle v^{- \ell(w_{S_0})}\ch([M_0]),\underline{H}_{w_{S_{0}}}\ch([N_0])\rangle_{\mathcal{H},\emptyset}\\
& = \left\langle v^{- \ell(w_{S_0})}\ch([M_0]),\underline{H}_{w_{S_{0}}}\overline{\sum_{w\in W_{S_0}\backslash W}b'_wH_{w_-}}\right\rangle_{\mathcal{H},\emptyset}\\
& = \left\langle v^{- \ell(w_{S_0})}\underline{H}_{w_{S_{0}}}\ch([M_0]),\overline{\sum_{w\in W_{S_0}\backslash W}b'_wH_{w_-}}\right\rangle_{\mathcal{H},\emptyset}\\
& = \left\langle\sum_{w\in W_{S_0}\backslash W,y\in W_{S_0}}a'_wv^{- \ell(y)}H_{yw_-},\overline{\sum_{w\in W_{S_0}\backslash W}b'_wH_{w_-}}\right\rangle_{\mathcal{H},\emptyset}\\
& = \sum_{w\in W_{S_0}\backslash W}\overline{a'_{w}b'_{w}} = \langle\presubscript{S_0}{\ch}([M]),\presubscript{S_0}{\ch}([N])\rangle.
\end{align*}
We get the theorem.
\end{proof}

\begin{prop}
For $M\in \Sbimod{S_0}$, we have $\presubscript{S_0}{\ch}(D(M))\simeq \overline{\presubscript{S_0}{\ch}(M)}$.
\end{prop}
\begin{proof}
Since $[\Sbimod{}]\to [\Sbimod{S_0}]$ defined by $[M]\mapsto [\pi_{S_0,\emptyset,*}(M)]$ is surjective, we may assume $M = \pi_{S_0,\emptyset,*}(M_0)$ for some $M_0\in \Sbimod{}$.
Let $p\colon \mathcal{H}\to \triv_{S_0}\otimes_{\mathcal{H}_{S_0}}\mathcal{H}$ be the map defined by $h\mapsto 1\otimes h$.
Then $\presubscript{S_0}{\ch}(D(M)) = p(\ch(D(M_0)))$ and $\overline{\presubscript{S_0}{\ch}(M)} = p(\overline{\ch(M_0)})$.
Therefore we may assume $S_0 = \emptyset$.

The $\Z[v,v^{-1}]$-algebra $[\Sbimod{}]\simeq \mathcal{H}$ is generated by $[B_s]$ with $s\in S$.
Hence we may assume $M = B_s$.
In this case, $D(B_s)\simeq B_s$ and $\overline{\ch(B_s)} = \overline{H_s} = H_s$.
Hence we get the proposition.
\end{proof}

\section{Parity sheaves and singular Soergel bimodules}
We continue to assume that $\Coeff$ is complete local Noetherian integral domain.

\subsection{General notation}
For an algebraic variety $X$ with an action of an algebraic group $B$, let $D_B^{\mathrm{b}}(X)$ be the bounded $B$-equivariant derived category of constructible $\Coeff$-coefficient sheaves.
Let $f\colon X\to Y$ be a morphism between algebraic varieties with $B$-actions and assume that $f$ commutes with the $B$-actions.
Then we have functors $f^!,f^*\colon D_B^{\mathrm{b}}(Y)\to D_B^{\mathrm{b}}(X)$ and $f_!,f_*\colon D_B^{\mathrm{b}}(X)\to D_B^{\mathrm{b}}(Y)$.
We also have the Verdier dual functor $D = D_X\colon D_B^{\mathrm{b}}(X)\to D_B^{\mathrm{b}}(X)$.
For $\mathcal{F}\in D_B^{\mathrm{b}}(X)$, the $n$-th $B$-equivariant cohomology of $\mathcal{F}$ is denoted by $H_B^n(X,\mathcal{F})$ and we put $H^{\bullet}_B(X,\mathcal{F}) = \bigoplus_n H_B^n(X,\mathcal{F})$.
We also put $\Hom^{\bullet}_{D_B^{\mathrm{b}}(X)}(\mathcal{F},\mathcal{G}) = \bigoplus_n\Hom_{D_B^{\mathrm{b}}(X)}(\mathcal{F},\mathcal{G}[n])$ for $\mathcal{F},\mathcal{G}\in D_B^{\mathrm{b}}(X)$.
The constant sheaf on $X$ is denoted by $\Coeff_X$ or just $\Coeff$.
The analogous notation applies for ind-varieties.

\subsection{Theorem}
Let $G$ be a Kac-Moody group over $\C$ attached to a generalized Cartan matrix.
We also have the Borel subgroup $B\subset G$, the unipotent radical $U\subset B$ and the Cartan subgroup $T\subset B$ such that $B = TU$.

Let $\Phi$ be the set of roots, $\Pi$ the set of simple roots and $W$ the Weyl group.
For each $\alpha\in\Phi$, we have the reflection $s_{\alpha}\in W$.
The subset $S = \{s_{\alpha}\mid \alpha\in\Pi\}$ gives a structure of a Coxeter system to $W$.
Let $X^*(T)$ be the character group of $T$ and set $V = X^*(T)\otimes_{\Z}\Coeff$.
For each $s = s_{\alpha}\in S$ with $\alpha\in\Pi$, we put $\alpha_s = \alpha$ and $\alpha_s^\vee = \alpha^\vee$.
Then with $(V,\{(\alpha_s,\alpha_s^\vee)\}_{s\in S})$, we have the category of Soergel bimodules $\Sbimod{}$.
We say that a subset $I\subset \Pi$ is of finite type if the subgroup of $W$ generated by $S_{I} = \{s_{\alpha}\mid \alpha\in I\}$ is finite.
We fix such $I$ throughout this section and put $\Sbimod{I} = \Sbimod{S_{I}}$, $W_{I} = W_{S_{I}}$.

With $I$, we have a parabolic subgroup $P_I \subset G$.
Let $\presubscript{I}{X} = P_I\backslash G$ be the generalized flag variety attached to $I$.
We also put $X = \presubscript{\emptyset}{X}$.
For each $w\in W_I\backslash W$, we have the Schubert variety $\presubscript{I}{X}_{\le w}\subset \presubscript{I}{X}$ and the Schubert cell $\presubscript{I}{X}_{w}\subset \presubscript{I}{X}_{\le w}$.
We set $\presubscript{I}{X}_{<w} = \presubscript{I}{X}_{\le w}\setminus \presubscript{I}{X}_{w}$.
Let $j_{\le w}$, $j_{<w}$ and $j_{w}$ be the inclusion maps $\presubscript{I}{X}_{\le w}\hookrightarrow \presubscript{I}{X}$, $\presubscript{I}{X}_{<w}\hookrightarrow \presubscript{I}{X}$ and $\presubscript{I}{X}_{w}\hookrightarrow \presubscript{I}{X}$, respectively.
If $J$ is a subset of $I$, we have the projection $\pi_{I,J}\colon \presubscript{J}{X}\to \presubscript{I}{X}$.
Let $\Parity_B(\presubscript{I}{X})\subset D^{\mathrm{b}}_B(\presubscript{I}{X})$ be the category of $B$-equivariant parity sheaves on $\presubscript{I}{X}$ with respect to the stratification by Schubert cells~\cite{MR3230821}.
For each $w\in W_I\backslash W$, there exists an indecomposable parity sheaf $\presubscript{I}{\mathcal{E}(w)}$ such that $\supp(\presubscript{I}{\mathcal{E}(w)})\subset \presubscript{I}{X}_{\le w}$ and $\presubscript{I}{\mathcal{E}(w)}|_{\presubscript{I}{X}_{w}}\simeq \Coeff[\ell(w)]$.
The functors $\pi_{I,J,*}$ and $\pi_{I,J}^*$ preserve the parity sheaves~\cite[Proposition~4.10]{MR3230821}.

Throughout this section, we assume the following.
\begin{itemize}
\item The torsion primes of $L_I$ are invertible in $\Coeff$. (See \cite[2.6]{MR3230821}.)
\item Let $\alpha,\beta$ be distinct positive roots of $L_I$.
Then $\{\alpha,\beta\}$ is linearly independent in $V/\mathfrak{m}V$ for any maximal ideal $\mathfrak{m}\subset\Coeff$.
\end{itemize}

Let $\mathcal{F}\in D^{\mathrm{b}}_B(\presubscript{I}{X})$ and $\mathcal{G}\in D^{\mathrm{b}}_B(X)$.
We define the convolution product $\mathcal{F}*\mathcal{G}\in D^{\mathrm{b}}_B(\presubscript{I}{X})$ as follows.
Let $p\colon G\to \presubscript{I}{X}$ be the natural projection and
\[
m\colon \presubscript{I}{X}\overset{B}{\times}G\to \presubscript{I}{X},\quad q\colon \presubscript{I}{X}\times G\to \presubscript{I}{X}\overset{B}{\times}G
\]
be the action map of $G$ on $X$ and the natural projection, respectively.
Then there exists unique $\mathcal{F}\overset{B}{\boxtimes}p^*\mathcal{G}\in D_B^{\mathrm{b}}(\presubscript{I}{X}\overset{B}{\times}G)$ such that $q^*(\mathcal{F}\overset{B}{\boxtimes}p^*\mathcal{G})\simeq \mathcal{F}\boxtimes p^*\mathcal{G}$.
Now we put $\mathcal{F}*\mathcal{G} = m_*(\mathcal{F}\overset{B}{\boxtimes}p^*\mathcal{G})$.
If $\mathcal{F}\in \Parity_B(\presubscript{I}{X})$ and $\mathcal{G}\in \Parity_B(X)$ then $\mathcal{F}*\mathcal{G}\in \Parity_B(\presubscript{I}{X})$~\cite[Theorem~4.8]{MR3230821}.

In this section, we prove the following.
We write $\pi_{I,J,*}$ (resp.\ $\pi_{I,J}^*$) for $\pi_{S_I,S_J,*}$ (resp.\ $\pi_{S_I,S_J}^*$) where $S_I = \{s_{\alpha}\mid \alpha\in I\}$.
Note that this notation is the same as the push-forward (resp.\ pull-back) with respect to $\pi_{I,J}\colon \presubscript{J}{X}\to \presubscript{I}{X}$.
The author thinks that the readers are not confused by this.
\begin{thm}\label{thm:equivalence}
There exists an equivalence of categories $\presubscript{I}{\mathbb{H}}\colon \Parity_B(\presubscript{I}{X})\to \Sbimod{I}$.
The functor satisfies the following.
\begin{enumerate}
\item For $\mathcal{F}\in \Parity_B(\presubscript{I}{X})$ and $\mathcal{G}\in \Parity_B(X)$, we have $\presubscript{I}{\mathbb{H}}(\mathcal{F}*\mathcal{G})\simeq \presubscript{I}{\mathbb{H}}(\mathcal{F})\otimes\mathbb{H}(\mathcal{G})$, here we put $\mathbb{H} = \presubscript{\emptyset}{\mathbb{H}}$.
\item For $J\subset I$, we have $\presubscript{I}{\mathbb{H}}\circ\pi_{I,J,*}\simeq \pi_{I,J,*}\circ\presubscript{J}{\mathbb{H}}$ and $\presubscript{J}{\mathbb{H}}\circ \pi_{I,J}^*\simeq \pi_{I,J}^*\circ\presubscript{I}{\mathbb{H}}$.
\item We have $D\circ\presubscript{I}{\mathbb{H}}\simeq \presubscript{I}{\mathbb{H}}\circ D$.
\end{enumerate}
\end{thm}

The functor $\presubscript{I}{\mathbb{H}}$ is given by taking the global sections.
We will give the definition in the next subsection.

\subsection{The functor $\presubscript{I}{\mathbb{H}}$}
Let $\mathcal{F}\in D_B^{\mathrm{b}}(\presubscript{I}{X})$ and we put $\presubscript{I}{\mathbb{H}}(\mathcal{F}) = H^{\bullet}_{B}(\presubscript{I}{X},\mathcal{F}) = \bigoplus_{n\in\Z}H_B^n(\presubscript{I}{X},\mathcal{F})$.
This is an $H^{\bullet}_{B}(\presubscript{I}{X})$-module.
Recall that $R = S(V) = S(X^*(T)\otimes_{\Z}\Coeff)$ and $R^{W_{I}}$ the subalgebra of $W_{I}$-fixed elements.
We have a natural homomorphism $R^{W_I}\otimes_{\Coeff}R \simeq H^{\bullet}_{P_I\times B}(\mathrm{pt})\to H^{\bullet}_{P_I\times B}(G)\simeq H^{\bullet}_B(\presubscript{I}{X})$.
Hence $\presubscript{I}{\mathbb{H}}(\mathcal{F})$ is an $(R^{W_{I}},R)$-bimodule.

Recall that $Q$ is a field of fractions of $R$.
Note that we have $H^{\bullet}_{B} = H^{\bullet}_{T}$.
By the localization theorem, 
\[
\presubscript{I}{\mathbb{H}}(\mathcal{F})\otimes_{R}Q 
=
H^{\bullet}_{T}(\presubscript{I}{X}^{T},\mathcal{F}|_{\presubscript{I}{X}^T})\otimes_{R}Q.
\]
The $T$-fixed points of $\presubscript{I}{X}^T$ are parametrized by $W_I\backslash W$.
For $w\in W_I\backslash W$, we use the same letter $w$ for the corresponding $T$-fixed point.
Then we have $\mathcal{F}|_{\presubscript{I}{X}^T} = \bigoplus_{w\in W_I\backslash W}\mathcal{F}_w$.
Therefore we get
\[
\presubscript{I}{\mathbb{H}}(\mathcal{F})\otimes_{R}Q
=
\bigoplus_{w\in W_{I}\backslash W}H_T^{\bullet}(\{w\},\mathcal{F}_w)\otimes_{R}Q.
\]
For $f\in R^{W_I}$ and $m\in H_T^{\bullet}(\{w\},\mathcal{F}_w)$, we have $fm = mw^{-1}(f)$.
Therefore by putting $\presubscript{I}{\mathbb{H}}(\mathcal{F})_Q^w = H_T^{\bullet}(\{w\},\mathcal{F}_w)\otimes_{R}Q$, we have $\presubscript{I}{\mathbb{H}}(\mathcal{F})\in \BigCat{I}$.

For $\mathcal{F}\in D_B^{\mathrm{b}}(\presubscript{I}{X})$, the complex $R\Gamma_{B}(\presubscript{I}{X},\mathcal{F})$ can be regarded as an $H^{\bullet}_B(\presubscript{I}{X})$-module.
Hence this is an $(R^{W_I},R)$-bimodule.
The proof of the following proposition is taken from \cite[Proposition~3.2.1]{MR3003920}.
\begin{prop}
Let $\mathcal{F}\in D_B^{\mathrm{b}}(\presubscript{I}{X})$ and $\mathcal{G}\in D_B^{\mathrm{b}}(X)$.
Then, as $(R^{W_I},R)$-bimodules, we have 
\[ R\Gamma_B(\presubscript{I}{X},\mathcal{F})\overset{L}{\otimes}_{R}R\Gamma_B(X,\mathcal{G})\simeq R\Gamma_B(\presubscript{I}{X},\mathcal{F}*\mathcal{G}).\]
\end{prop}
\begin{proof}
Set $\widetilde{X} = U\backslash G$.
Then $T$ acts on $\widetilde{X}$ from the left.
Consider the action of $T\times T$ on $\presubscript{I}{X}\times \widetilde{X}$ defined by $(t_1,t_2)(x,y) = (xt_1^{-1},t_2x)$.
The action of $\diag(T)\subset T\times T$ is free and the quotient space is $\presubscript{I}{X}\overset{T}{\times} \widetilde{X}$.
On this space, we have an action of $(T\times T)/\diag(T)$.

Let $p\colon \widetilde{X}\to X$ be the natural projection.
Then there exists a unique object $\mathcal{F}\overset{T}{\boxtimes}p^*\mathcal{G}\in D_{(T\times T)/\diag(T)\times B}^{\mathrm{b}}(\presubscript{I}{X}\overset{T}{\times} \widetilde{X})$ such that $q^*(\mathcal{F}\overset{T}{\boxtimes}\mathcal{G})\simeq \mathcal{F}\boxtimes p^*(\mathcal{G})$ where $q\colon \presubscript{I}{X}\times \widetilde{X}\to \presubscript{I}{X}\overset{T}{\times} \widetilde{X}$ is the natural projection.
By \cite[Corollary~B.4.2]{MR3003920}, we have
\[
R\Gamma_{(T\times T)/\diag(T)\times B}(\presubscript{I}{X}\overset{T}{\times} \widetilde{X},\mathcal{F}\overset{T}{\boxtimes}p^*\mathcal{G})\overset{L}{\otimes}_{H^{\bullet}_{T\times T/\diag(T)}(\mathrm{pt})}\Coeff\simeq R\Gamma_B(\presubscript{I}{X}\overset{T}{\times} \widetilde{X},\mathcal{F}\overset{T}{\boxtimes}p^*\mathcal{G}),
\]
here $B$ acts on $\widetilde{X}$ from the right.
The left hand side is isomorphic to 
\[
R\Gamma_{T\times T\times B}(\presubscript{I}{X}\times \widetilde{X},\mathcal{F}\boxtimes p^*\mathcal{G})\overset{L}{\otimes}_{H^{\bullet}_{T\times T/\diag(T)}(\mathrm{pt})}\Coeff\simeq R\Gamma^{\bullet}_T(\presubscript{I}{X},\mathcal{F})\overset{L}{\otimes}_{R} R\Gamma_{T\times B}(\widetilde{X},p^*\mathcal{G})
\]
and
\[
R\Gamma_{T\times B}(\widetilde{X},p^*\mathcal{G})\simeq R\Gamma_{B}(X,\mathcal{G})
\]
since the action of $T$ on $\widetilde{X}$ is free.
Hence 
\[
R\Gamma_B(\presubscript{I}{X},\mathcal{F})\overset{L}{\otimes}_{R}R\Gamma_B(X,\mathcal{G})\simeq R\Gamma_B(\presubscript{I}{X}\overset{T}{\times} \widetilde{X},\mathcal{F}\overset{T}{\boxtimes}p^*\mathcal{G}).
\]
Both projections $\presubscript{I}{X}\overset{T}{\times}G\to \presubscript{I}{X}\overset{T}{\times}\widetilde{X}$ and $\presubscript{I}{X}\overset{T}{\times}G\to \presubscript{I}{X}\overset{B}{\times}G$ are fibrations such that fibers are isomorphic to pro-affine spaces.
Hence
\[
R\Gamma_B(\presubscript{I}{X}\overset{T}{\times}\widetilde{X},\mathcal{F}\overset{T}{\boxtimes}p^*\mathcal{G})\simeq R\Gamma_B(\presubscript{I}{X}\overset{B}{\times}G,\mathcal{F}\overset{B}{\boxtimes}p^*\mathcal{G}).
\]
By the definition of the convolution, the last one is isomorphic to $R\Gamma_B(\presubscript{I}{X},\mathcal{F}*\mathcal{G})$.
\end{proof}

For $s\in S$, let $j_{\le s}\colon X_{\le s}\hookrightarrow X$ be the inclusion map.
Since $X_{\le s}\simeq \mathbb{P}^1$, it is easy to see that $j_{\le s,*}\Coeff[1]$ is an indecomposable parity sheaf.
Therefore it is isomorphic to $\mathcal{E}(s)$.
Hence we have $\mathbb{H}(\mathcal{E}(s))\simeq H^{\bullet}_{B}(X_{\le s},\Coeff[1]) \simeq B_s$.
Since this is free as a left $R$-module, with the above proposition, we get $\mathbb{H}(\mathcal{F}*\mathcal{E}(s))\simeq \mathbb{H}(\mathcal{F})\otimes_{R}B_s$ as an $(R^{W_I},R)$-bimodules for any $\mathcal{F}\in D_B^{\mathrm{b}}(\presubscript{I}{X})$.
\begin{lem}\label{lem:monoidal structure and H}
We have $\presubscript{I}{\mathbb{H}}(\mathcal{F}*\mathcal{E}(s))\simeq \presubscript{I}{\mathbb{H}}(\mathcal{F})\otimes B_s$ as objects in $\BigCat{I}$.
\end{lem}
\begin{proof}
Set $\widetilde{X} = U\backslash G$.
Let $\mathcal{F}\in D_B^{\mathrm{b}}(\presubscript{I}{X})$ and $\mathcal{G}\in D_B^{\mathrm{b}}(X)$.
For $x\in W_{I}\backslash W$ and $y\in W$, we fix representatives in $G$ and denote the representatives by the same letter $x,y$.
Let $p\colon G\to X$, $q\colon \widetilde{X}\to X$ be the natural projections and 
$Y$ 
the inverse image of $xy\in \presubscript{I}{X}$ by the action map $\presubscript{I}{X}\overset{B}{\times}G\to \presubscript{I}{X}$.
Then we have 
\[
R\Gamma_T(\{xy\},(\mathcal{F}*\mathcal{G})_{xy}) \simeq R\Gamma_T(Y,(\mathcal{F}\overset{B}{\boxtimes}p^*\mathcal{G})|_{Y})
\]
Let $z_0$ be the image of $(P_Ix,y)$ in $\presubscript{I}{X}\overset{B}{\times}G$.
Then $z_0\in Y$ and it is fixed by $T$.
Hence we have a natural map
\[
R\Gamma_T(Y,(\mathcal{F}\overset{B}{\boxtimes}p^*\mathcal{G})|_{Y})
\to
R\Gamma_T(\{z_0\},(\mathcal{F}\overset{B}{\boxtimes}p^*\mathcal{G})_{z_0}).
\]
Denote the inverse image of $z_0$ under $\presubscript{I}{X}\overset{T}{\times}G\to \presubscript{I}{X}\overset{B}{\times}G$ by $Z_2$ and the image of $Z_2$ under $\presubscript{I}{X}\overset{T}{\times}\widetilde{X}$ by $Z_1$.
Then $U\simeq Z_2$ by $u\mapsto [(P_Ixu^{-1},uy)]$ and $Z_1$ is the image of $\presubscript{I}{X}_x\times \{Uy\}\subset \presubscript{I}{X}\times \widetilde{X}$ in $\presubscript{I}{X}\overset{T}{\times}\widetilde{X}$.
Therefore $Z_2\to \{z_0\}$ and $Z_2\to Z_1$ are fibrations whose fibers are isomorphic to pro-affine spaces.
Hence 
\[
R\Gamma_T(\{z_0\},(\mathcal{F}\overset{B}{\boxtimes}p^*\mathcal{G})_{z_0})
\simeq
R\Gamma_T(Z_1,(\mathcal{F}\overset{T}{\boxtimes}q^*\mathcal{G})|_{Z_1})
\]
Let $z_1$ be the image of $(P_Ix,Uy)\in \presubscript{I}{X}\times \widetilde{X}$ in $\presubscript{I}{X}\overset{T}{\times }\widetilde{X}$.
This is a $(T\times T)/\diag(T)$-fixed point and contained in $Z_1$.
Hence we have a natural morphism
\[
R\Gamma_T(Z_1,(\mathcal{F}\overset{T}{\boxtimes}q^*\mathcal{G})|_{Z_1})
\to
R\Gamma_T(\{z_1\},(\mathcal{F}\overset{T}{\boxtimes}q^*\mathcal{G})_{z_1})
\]
As in the proof of the above proposition, we have
\[
R\Gamma_T(\{z_1\},(\mathcal{F}\overset{T}{\boxtimes}q^*\mathcal{G})_{z_1})
\simeq 
R\Gamma_{((T\times T)/\diag(T))\times T}(\{z_1\},(\mathcal{F}\overset{T}{\boxtimes}q^*\mathcal{G})_{z_1})\overset{L}{\otimes}_{H^{\bullet}_{(T\times T)/\diag T}(\mathrm{pt})}\Coeff.
\]
Let $Z$ be the inverse image of $z_1$ by $\presubscript{I}{X}\times \widetilde{X}\to \presubscript{I}{X}\overset{T}{\times}\widetilde{X}$.
Then $Z\simeq \{P_Ix\}\times Z_0$ where $Z_0 = \{Uty\mid t\in T\}$ and $Z\to \{z_1\}$ is a $T$-torsor.
We have
\begin{align*}
& R\Gamma_{((T\times T)/\diag(T))\times T}(\{z_1\},(\mathcal{F}\overset{T}{\boxtimes}q^*\mathcal{G})_{z_1})\overset{L}{\otimes}_{H^{\bullet}_{(T\times T)/\diag T}(\mathrm{pt})}\Coeff\\
& \simeq R\Gamma_{T\times T\times T}(Z,(\mathcal{F}\boxtimes q^*\mathcal{G})|_Z)\overset{L}{\otimes}_{H^{\bullet}_{(T\times T)/\diag T}(\mathrm{pt})}\Coeff\\
& \simeq R\Gamma_T(\{P_Ix\},\mathcal{F}_{P_Ix})\overset{L}{\otimes}_{R}R\Gamma_{T\times T}(Z_0,q^*\mathcal{G}_{Z_0})\\
& \simeq R\Gamma_T(\{P_Ix\},\mathcal{F}_{P_Ix})\overset{L}{\otimes}_{R}R\Gamma_{T}(\{By\},\mathcal{G}_{By})
\end{align*}
since $T$ acts on $Z_0$ freely.
Hence we get a map $R\Gamma_T(\{xy\},(\mathcal{F}*\mathcal{G})_{xy})\to R\Gamma_T(\{x\},\mathcal{F}_{x})\overset{L}{\otimes}_{R}R\Gamma_{T}(\{y\},\mathcal{G}_{y})$.
By the construction, the following diagram is commutative:
\[
\begin{tikzcd}
R\Gamma_B(\presubscript{I}{X},\mathcal{F}*\mathcal{G})\arrow[r,dash,"\sim"]\arrow[d] & R\Gamma_B(\presubscript{I}{X},\mathcal{F})\overset{L}{\otimes}_{R}R\Gamma_B(X,\mathcal{G})\arrow[d]\\
R\Gamma_T(\{xy\},(\mathcal{F}*\mathcal{G})_{xy})\arrow[r] & R\Gamma_T(\{x\},\mathcal{F}_{x})\overset{L}{\otimes}_{R}R\Gamma_{T}(\{y\},\mathcal{G}_{y}).
\end{tikzcd}
\]
Now let $\mathcal{G} = \mathcal{E}(s)$.
Then both $\mathbb{H}(\mathcal{G})$ and $H^{\bullet}_{T}(\{y\},\mathcal{G}_{y})$ are free as left $R$-modules.
Therefore, by taking the cohomology and tensoring $Q$, we get the following commutative diagram:
\[
\begin{tikzcd}
\presubscript{I}{\mathbb{H}}(\mathcal{F}*\mathcal{G})\otimes_{R}Q\arrow[d]\arrow[r,dash,"\sim"] & (\presubscript{I}{\mathbb{H}}(\mathcal{F})\otimes_{R}Q)\otimes_{Q}(\mathbb{H}(\mathcal{G})\otimes_{R}Q)\arrow[d]\\
\presubscript{I}{\mathbb{H}}(\mathcal{F}*\mathcal{G})_Q^{xy}\arrow[r] & \presubscript{I}{\mathbb{H}}(\mathcal{F})^{x}_Q\otimes_R\mathbb{H}(\mathcal{G})^y_Q.
\end{tikzcd}
\]
Therefore, for each $w\in W$, we have
\[
\begin{tikzcd}
\presubscript{I}{\mathbb{H}}(\mathcal{F}*\mathcal{G})\otimes_{R}Q\arrow[d]\arrow[r,dash,"\sim"] & (\presubscript{I}{\mathbb{H}}(\mathcal{F})\otimes_{R}Q)\otimes_{Q}(\mathbb{H}(\mathcal{G})\otimes_{R}Q)\arrow[d]\\
\presubscript{I}{\mathbb{H}}(\mathcal{F}*\mathcal{G})_Q^{w}\arrow[r] & \bigoplus_{xy = w}\presubscript{I}{\mathbb{H}}(\mathcal{F})^{x}_Q\otimes_Q\mathbb{H}(\mathcal{G})^y_Q\arrow[r,equal] & (\presubscript{I}{\mathbb{H}}(\mathcal{F})\otimes\mathbb{H}(\mathcal{G}))_Q^w.
\end{tikzcd}
\]
Hence the isomorphism $\presubscript{I}{\mathbb{H}}(\mathcal{F}*\mathcal{E}(s))\simeq \presubscript{I}{\mathbb{H}}(\mathcal{F})\otimes_{R}B_s$ is an isomorphism in $\BigCat{I}$.
\end{proof}

\begin{prop}\label{prop:H and pi_*}
Let $J\subset I$ be a subset.
Then we have $\presubscript{I}{\mathbb{H}}\circ\pi_{I,J,*}\simeq \pi_{I,J,*}\circ\presubscript{J}{\mathbb{H}}$.
\end{prop}
\begin{proof}
Let $\mathcal{F}\in D_B^{\mathrm{b}}(\presubscript{J}{X})$.
Since $H^n_B(\presubscript{I}{X},\pi_{I,J,*}\mathcal{F})\simeq H_B^n(\presubscript{J}{X},\mathcal{F})$, we have $\presubscript{I}{\mathbb{H}(\pi_{I,J,*}(\mathcal{F}))}\simeq \pi_{I,J,*}(\presubscript{J}{\mathbb{H}}(\mathcal{F}))$ as $(R^{W_I},R)$-bimodules.
For each $w\in W_I\backslash W$, by the localization theorem, we have
\begin{align*}
H^n_T(\{w\},\pi_{I,J,*}(\mathcal{F})_{w})\otimes_{R}Q & \simeq H^n_T(\pi_{I,J}^{-1}(\{w\}),\mathcal{F}|_{\pi_{I,J}^{-1}(w)})\otimes_{R}Q\\
& \simeq H_T^n((\pi_{I,J}^{-1}(\{w\}))^{T},\mathcal{F}|_{(\pi_{I,J}^{-1}(\{w\}))^{T}})\otimes_{R}Q.
\end{align*}
Since $(\pi_{I,J}^{-1}(\{w\}))^{T} = \{x\in W_{J}\backslash W\mid \overline{x} = w\}$ where $\overline{x}$ is the image of $x$ in $W_I\backslash W$, we have
\begin{align*}
H_T^n((\pi_{I,J}^{-1}(\{w\}))^{T},\mathcal{F}|_{(\pi_{I,J}^{-1}(\{w\}))^{T}})\otimes_{R}Q
& \simeq
\bigoplus_{x\in W_{J}\backslash W,\ \overline{x} = w}
H_T^n(\{x\},\mathcal{F}|_{\{x\}})\otimes_{R}Q\\
& =
(\pi_{I,J,*}\presubscript{I}{\mathbb{H}}(\mathcal{F}))_Q^w.
\end{align*}
Therefore $\presubscript{I}{\mathbb{H}(\pi_{I,J,*}(\mathcal{F}))}\simeq \pi_{I,J,*}(\presubscript{J}{\mathbb{H}}(\mathcal{F}))$ in $\BigCat{I}$.
\end{proof}

\begin{cor}
We have $\presubscript{I}{\mathbb{H}}(\Parity_B(\presubscript{I}{X}))\subset \Sbimod{I}$.
\end{cor}
\begin{proof}
First we assume that $I = \emptyset$.
Any object in $\Parity_B(X)$ is a direct summand of a direct sum of objects of a form $\mathcal{E}(s_1)*\dotsm*\mathcal{E}(s_l)[n]$ where $s_1,\dots,s_l\in S$ and $n\in\Z$.
Hence the corollary follows from Lemma~\ref{lem:monoidal structure and H}.

In general, let $\mathcal{F}\in \Parity_B(\presubscript{I}{X})$.
We may assume that $\mathcal{F}\simeq \presubscript{I}{\mathcal{E}(w)}$ for some $w\in W_I\backslash W$.
The object $\pi_{I,\emptyset,*}(\mathcal{E}(w_-))$ is a parity sheaf.
By analyzing the support, $\pi_{I,\emptyset,*}(\mathcal{E}(w_-))$ contains $\presubscript{I}{\mathcal{E}(w)}[l]$ as a direct summand for some $l\in \Z$.
Hence $\presubscript{I}{\mathbb{H}}(\presubscript{I}{\mathcal{E}(w)})$ is a direct summand of $\presubscript{I}{\mathbb{H}}(\pi_{I,\emptyset,*}(\mathcal{E}(w_-)[-l]))\simeq \pi_{I,\emptyset,*}(\presubscript{I}{\mathbb{H}}(\mathcal{E}(w_-)[-l])))$.
Since $\presubscript{I}{\mathbb{H}}(\mathcal{E}(w_-)[-l])\in \Sbimod{}$, we have $\presubscript{I}{\mathbb{H}}(\presubscript{I}{\mathcal{E}(w)})\in \Sbimod{I}$.
\end{proof}

\begin{cor}
For $\mathcal{F}\in \Parity_B(\presubscript{I}{X})$ and $\mathcal{G}\in\Parity_B(X)$, we have $\presubscript{I}{\mathbb{H}}(\mathcal{F}*\mathcal{G})\simeq \presubscript{I}{\mathbb{H}}(\mathcal{F})\otimes\mathbb{H}(\mathcal{G})$.
\end{cor}
\begin{proof}
Since $\mathbb{H}(\mathcal{G})\in\Sbimod{}$ is free as a left $R$-module, the same proof of Lemma~\ref{lem:monoidal structure and H} can apply.
\end{proof}

\begin{lem}
For $\mathcal{F}\in\Parity_B(\presubscript{I}{X})$, we have $\presubscript{I}{\mathbb{H}}(D(\mathcal{F}))\simeq D(\presubscript{I}{\mathbb{H}}(\mathcal{F}))$.
\end{lem}
\begin{proof}
This follows the equivariant Poincar\'e duality and the freeness of $\presubscript{I}{\mathbb{H}}(\mathcal{F})$ as a right $R$-module.
\end{proof}

\subsection{Proof of Theorem~\ref{thm:equivalence}}
Let $Z$ be a closed subset of $\presubscript{I}{X}$ which is a union of Schubert cells and $U = \presubscript{I}{X}\setminus Z$.
By putting $\mathcal{F}$ to be the constant sheaf in \cite[Corollary~2.9]{MR3230821}, for $j\colon U\hookrightarrow \presubscript{I}{X}$ and $i\colon Z\hookrightarrow \presubscript{I}{X}$, the sequence
\[
0\to \presubscript{I}{\mathbb{H}}(i_*i^!\mathcal{G})\to \presubscript{I}{\mathbb{H}}(\mathcal{G})\to \presubscript{I}{\mathbb{H}}(j_*j^!\mathcal{G})\to 0
\]
is exact for a $!$-parity sheaf $\mathcal{G}\in D_B^{\mathrm{b}}(\presubscript{I}{X})$.
\begin{rem}
\cite[Corollary~2.9]{MR3230821} only states that $\presubscript{I}{\mathbb{H}}(\mathcal{G})\to \presubscript{I}{\mathbb{H}}(j_*j^!\mathcal{G})$ is surjective.
The long exact sequence attached to the triangle $i_*i^!\mathcal{G}\to \mathcal{G}\to j_*j^!\mathcal{G}\xrightarrow{+1}$ implies that the kernel is isomorphic to $\presubscript{I}{\mathbb{H}}(i_*i^!\mathcal{G})$.
\end{rem}

Since $\mathcal{G}$ is a $!$-parity sheaf, $j^!\mathcal{G}$ and $i^!\mathcal{G}$ are $!$-parity.
Hence by \cite[Proposition~2.6]{MR3230821}, we have a (non-canonical) isomorphism $\presubscript{I}{\mathbb{H}}(i_*i^!\mathcal{G}) \simeq \bigoplus_{w\in W_{I}\backslash W,\presubscript{I}{X}_w\subset Z}H^{\bullet}_B(j_w^!\mathcal{G})$.
Since $j_w^!\mathcal{G}$ is isomorphic to a direct sum of shifts of the constant sheaf, $H^{\bullet}_B(j_w^!\mathcal{G})$ is graded free as an $H^{\bullet}_B(\mathrm{pt})\simeq R$-module.
Hence $\presubscript{I}{\mathbb{H}}(i_*i^!\mathcal{G})$ is graded free.
By the same argument, $\presubscript{I}{\mathbb{H}}(j_*j^!\mathcal{G})$ is also graded free.

We apply the above argument to $\mathcal{G} = D(\mathcal{F})$ where $\mathcal{F}\in D_B^{\mathrm{b}}(\presubscript{I}{X})$ is a $*$-parity sheaf.
Since $\presubscript{I}{\mathbb{H}}(D(\mathcal{F}))$ is free, we have $D(\presubscript{I}{\mathbb{H}}(D(\mathcal{F})))\simeq \presubscript{I}{\mathbb{H}}(D^2(\mathcal{F}))\simeq \presubscript{I}{\mathbb{H}}(\mathcal{F})$.
Hence $\presubscript{I}{\mathbb{H}}(\mathcal{F})$ is graded free.
Similarly we have $D(\presubscript{I}{\mathbb{H}}(i_*i^!D(\mathcal{F})))\simeq \presubscript{I}{\mathbb{H}}(i_!i^*\mathcal{F})$ and $D(\presubscript{I}{\mathbb{H}}(j_*j^!D(\mathcal{F})))\simeq \presubscript{I}{\mathbb{H}}(j_!j^*\mathcal{F})$ and both are graded free.
Therefore we get an exact sequence
\begin{equation}\label{eq:exact seqence in open and closed}
0\to \presubscript{I}{\mathbb{H}}(j_!j^*\mathcal{F})\to \presubscript{I}{\mathbb{H}}(\mathcal{F})\to \presubscript{I}{\mathbb{H}}(i_!i^*\mathcal{F})\to 0
\end{equation}
and each term is a graded free $R$-module.

\begin{lem}\label{lem:general lemma to determine M_I,M^I}
Let $M_1,M,M_2\in \BigCat{I}$ with a sequence $0\to M_1\to M\to M_2\to 0$ in $\BigCat{I}$ which is exact as $(R^{W_I},R)$-bimodules.
Let $A\subset W_I\backslash W$ and assume that $\supp_{W}(M_1)\subset A$, $\supp_{W}(M_2)\subset (W_I\backslash W)\setminus A$ and $M_2$ is torsion-free as a right $R$-module.
Then we have $M_1\simeq M_A$ and $M_2\simeq M^{(W_I\backslash W)\setminus A}$.
\end{lem}
\begin{proof}
We regard $M_1$ as a submodule of $M$.
Then $M_1\subset M_A$ is obvious.
Let $m\in M_A$.
Then the image of $m$ in $M_2\otimes_{R}Q$ is zero.
By assumption, $m = 0$ in $M_2$.
Hence $m\in M_1$.
\end{proof}

\begin{lem}\label{lem:relation of restriction between sheaf and Soergel bimodule}
Keep the above notation and let $A\subset W_I\backslash W$ be the subset such that $Z = \bigcup_{w\in A}\presubscript{I}{X}_w$.
Set $A^c = (W_{I}\backslash W)\setminus A$.
\begin{enumerate}
\item If $\mathcal{F}\in D_B^{\mathrm{b}}(\presubscript{I}{X})$ is $*$-parity, $\presubscript{I}{\mathbb{H}}(j_!j^*\mathcal{F})\simeq \presubscript{I}{\mathbb{H}}(\mathcal{F})_{A^c}$ and $\presubscript{I}{\mathbb{H}}(i_!i^*\mathcal{F})\simeq \presubscript{I}{\mathbb{H}}(\mathcal{F})^{A}$.
\item If $\mathcal{F}\in D_B^{\mathrm{b}}(\presubscript{I}{X})$ is $!$-parity, $\presubscript{I}{\mathbb{H}}(j_*j^!\mathcal{F})\simeq \presubscript{I}{\mathbb{H}}(\mathcal{F})^{A^c}$ and $\presubscript{I}{\mathbb{H}}(i_*i^!\mathcal{F})\simeq \presubscript{I}{\mathbb{H}}(\mathcal{F})_{A}$.
\end{enumerate}
\end{lem}
\begin{proof}
We have $\presubscript{I}{\mathbb{H}}(\mathcal{F})_Q^w = H^{\bullet}_T(\{w\},(j_!j^*\mathcal{F})_w)\otimes_{R}Q$ and it is zero if $w\notin A^c$.
Therefore $\supp_W(\presubscript{I}{\mathbb{H}}(j_!j^*\mathcal{F}))\subset A^c$.
Similarly we also have $\supp_W(\presubscript{I}{\mathbb{H}}(i_!i^*\mathcal{F}))\subset A$.
Therefore, by the exact sequence \eqref{eq:exact seqence in open and closed} and since $\presubscript{I}{\mathbb{H}}(i_!i^*\mathcal{F})$ is a free (hence tosion-free) $R$-module, we get (1) by Lemma~\ref{lem:general lemma to determine M_I,M^I}.

We prove (2).
If $w\notin A$, then $(i_*i^!\mathcal{F})_w = 0$.
Hence $\presubscript{I}{\mathbb{H}}(i_*i^!\mathcal{F})_Q^w = 0$.
Therefore we have $\supp_W(\presubscript{I}{\mathbb{H}}(i_*i^!\mathcal{F}))\subset A$.
Let $w\in A$.
We have $\dim_Q\presubscript{I}{\mathbb{H}}(j_*j^*\mathcal{F})^w_Q = \dim_Q D(\presubscript{I}{\mathbb{H}}(j_*j^*\mathcal{F}))^w_Q = \dim_Q\presubscript{I}{\mathbb{H}}(j_!j^!(D(\mathcal{F})))^w_Q = 0$.
Hence $\supp_W(\presubscript{I}{\mathbb{H}}(j_*j^*\mathcal{F}))\subset A^c$.
We get (2) with the previous lemma.
\end{proof}

\begin{lem}\label{lem:relation of restriction between sheaf and Soergel bimodule, locally closed}
Let $Z' = \bigcup_{w\in A'}\presubscript{I}{X}_w\subset Z$ be a closed subset and set $U' = \presubscript{I}{X}\setminus Z$.
Put $Y = Z\cap U'$ and denote the inclusions $Z\hookrightarrow\presubscript{I}{X}$, $U'\hookrightarrow \presubscript{I}{X}$, $Y\hookrightarrow \presubscript{I}{X}$ by $i'$, $j'$, $a$, respectively.
Let $\mathcal{F}\in D_B^{\mathrm{b}}(\presubscript{I}{X})$.
\begin{enumerate}
\item If $\mathcal{F}$ is $*$-parity, then $0\to \presubscript{I}{\mathbb{H}}(j_!j^*\mathcal{F})\to \presubscript{I}{\mathbb{H}}(j'_!(j')^*\mathcal{F})\to \presubscript{I}{\mathbb{H}}(a_!a^*\mathcal{F})\to 0$ is exact.
\item If $\mathcal{F}$ is $!$-parity, then $0\to \presubscript{I}{\mathbb{H}}(i'_*(i')^!\mathcal{F})\to \presubscript{I}{\mathbb{H}}(i_*i^!\mathcal{F})\to \presubscript{I}{\mathbb{H}}(a_*a^!\mathcal{F})\to 0$ is exact.
\end{enumerate}
\end{lem}
\begin{proof}
Note that $j'_{!}(j')^*\mathcal{F}$ is $*$-parity if $\mathcal{F}$ is $*$-parity.
By \eqref{eq:exact seqence in open and closed}, we have an exact sequence $0\to \presubscript{I}{\mathbb{H}}(j_!j^*j'_!(j')^*\mathcal{F})\to \presubscript{I}{\mathbb{H}}(j_!j^*\mathcal{F})\to \presubscript{I}{\mathbb{H}}(i_!i^*j'_!(j')^*\mathcal{F})\to 0$.
We have $j_!j^*j'_!(j')^*\simeq j_!j^*$ and $i_!i^*j'_!(j')^*\simeq a_!a^*$.
We deduce (1), and (2) follows from a similar argument.
\end{proof}

In particular, $\presubscript{I}{\mathbb{H}}(j_{w!}j_w^*\mathcal{F})\simeq \presubscript{I}{\mathbb{H}}(\mathcal{F})_{\ge w}/\presubscript{I}{\mathbb{H}}(\mathcal{F})_{> w}$ if $\mathcal{F}\in \Parity_{B}(\presubscript{I}{X})$.

\begin{lem}\label{lem:on Schubert cell and ch}
Let $\mathcal{F}\in\Parity_B(\presubscript{I}{X})$ and $w\in W_I\backslash W$.
\begin{enumerate}
\item Take $n_{w,k}\in\Z_{\ge 0}$ such that $j_w^*\mathcal{F}\simeq \bigoplus_k\Coeff[k]^{\oplus n_{w,k}}$.
Then the coefficient of $1\otimes H_{w_-}$ in $\presubscript{I}{\ch}(\presubscript{I}{\mathbb{H}}(\mathcal{F}))$ is $v^{-\ell(w_-)}\sum_k n_{w,k}v^{k}$.
\item Take $m_{w,k}\in\Z_{\ge 0}$ such that $j_w^!\mathcal{F}\simeq \bigoplus_k\Coeff[k]^{\oplus m_{w,k}}$.
Then the coefficient of $1\otimes H_{w_-}$ in $\overline{\presubscript{I}{\ch}(\presubscript{I}{\mathbb{H}}(\mathcal{F}))}$ is $v^{\ell(w_-)}\sum m_{w,k}v^{-k}$.
\end{enumerate}
\end{lem}
\begin{proof}
We prove (2) first.
We have
\begin{align*}
D(\presubscript{I}{\mathbb{H}}(j_{w!}j_{w}^*(D(\mathcal{F})))) & \simeq \presubscript{I}{\mathbb{H}}(j_{w*}j_{w}^!\mathcal{F})\\
& \simeq \bigoplus_k H_B^{\bullet}(\presubscript{I}{X}_{w},\Coeff[k]^{\oplus m_{w,k}})\\
& \simeq \bigoplus_k R(k)^{\oplus m_{w,k}}.
\end{align*}
By the previous lemma, we have $\presubscript{I}{\mathbb{H}}(j_{w!}j_{w}^*\mathcal{F})\simeq \presubscript{\mathbb{H}}(\mathcal{F})_{\ge w}/\presubscript{\mathbb{H}}(\mathcal{F})_{> w}$.
Hence the coefficient of $1\otimes H_{w_-}$ in $\presubscript{I}{\ch}(\presubscript{I}{\mathbb{H}}(D(\mathcal{F})))$ is $v^{\ell(w_-)}\sum_k m_{w,k}v^{-k}$.
As we have $\presubscript{I}{\ch}(\presubscript{I}{\mathbb{H}}(D(\mathcal{F}))) = \overline{\presubscript{I}{\ch}(\presubscript{I}{\mathbb{H}}(\mathcal{F}))}$, we get (2).

We prove (1).
We have $j_w^!D(\mathcal{F})\simeq \bigoplus_k D(\Coeff_{\presubscript{I}{X}_{w}})[-k]^{\oplus n_{w,k}}$.
Since $\presubscript{I}{X}_{w}\simeq \mathbb{A}^{\ell(w_-)}$, we have $D(\Coeff_{\presubscript{I}{X}_{w}})\simeq \Coeff_{\presubscript{I}{X}_{w}}[2\ell(w_-)]$.
Hence $j_w^!D(\mathcal{F})\simeq \bigoplus_k \Coeff_{\presubscript{I}{X}_{w}}[2\ell(w_-) - k]^{\oplus n_{w,k}}$.
By (2), the coefficient of $H_w$ in $\overline{\presubscript{I}{\ch}(\presubscript{I}{\mathbb{H}}(D(\mathcal{F})))} = \presubscript{I}{\ch}(\presubscript{I}{\mathbb{H}}(\mathcal{F}))$ is $v^{\ell(w_-)}\sum_k n_{w,2\ell(w_-) - k}v^{-k} = v^{-\ell(w_-)}\sum_k n_{w,k}v^{k}$.
\end{proof}

\begin{lem}\label{lem:hom formula on Parity}
For any $\mathcal{F},\mathcal{G}\in \Parity_B(\presubscript{I}{X})$, $\Hom^{\bullet}_{\Parity_B(\presubscript{I}{X})}(\mathcal{F},\mathcal{G})$ is graded free and we have
\[
\grk\Hom^{\bullet}_{\Parity_B(\presubscript{I}{X})}(\mathcal{F},\mathcal{G}) = \grk\Hom_{\Sbimod{I}}^{\bullet}(\presubscript{I}{\mathbb{H}}(\mathcal{F}),\presubscript{I}{\mathbb{H}}(\mathcal{G})).
\]
\end{lem}
\begin{proof}
Take $n_{w,k}\in\Z_{\ge 0}$ (resp.\ $m_{w,k}\in\Z_{\ge 0}$) such that $j_w^*\mathcal{F}\simeq \bigoplus_k \Coeff[k]^{\oplus n_{w,k}}$ (resp.\ $j_w^!\mathcal{G}\simeq \bigoplus_k \Coeff[k]^{\oplus m_{w,k}}$).
By \cite[Proposition~2.6]{MR3230821}, we have
\begin{align*}
\Hom^{\bullet}_{\Parity_B(\presubscript{I}{X})}(\mathcal{F},\mathcal{G}) & \simeq \bigoplus_w\Hom_{D_B^{\mathrm{b}}(\presubscript{I}{X}_w)}(j_w^*\mathcal{F},j_w^!\mathcal{G})\\
& \simeq \bigoplus_w \bigoplus_{k,l}\Hom^{\bullet}_{D_B^{\mathrm{b}}(\presubscript{I}{X}_w)}(\Coeff[k]^{\oplus n_{w,k}},\Coeff[l]^{\oplus m_{w,l}})\\
& = \bigoplus_w \bigoplus_{k,l}R(-k + l)^{\oplus n_{w,k}m_{w,l}}.
\end{align*}
This is graded free and, by Theorem~\ref{thm:hom formula} and Lemma~\ref{lem:on Schubert cell and ch}, the graded rank is equal to $\grk\Hom^{\bullet}_{\Sbimod{I}}(\presubscript{I}{\mathbb{H}}(\mathcal{F}),\presubscript{I}{\mathbb{H}}(\mathcal{G}))$.
\end{proof}

Now we are ready to prove that $\presubscript{I}{\mathbb{H}}$ is fully-faithful.
\begin{lem}
Let $\Coeff'$ be a $\Coeff$-algebra.
\begin{enumerate}
\item Let $\mathcal{F},\mathcal{G}\in \Parity_B(\presubscript{I}{X})$.
Then $\Coeff'\otimes_{\Coeff}\Hom(\mathcal{F},\mathcal{G})\simeq \Hom(\Coeff'\otimes_{\Coeff}\mathcal{F},\Coeff'\otimes_{\Coeff}\mathcal{G})$.
\item Let $M,N\in \Sbimod{I}$.
Then $\Coeff'\otimes_{\Coeff}\Hom(M,N)\simeq \Hom(\Coeff'\otimes_{\Coeff}M,\Coeff'\otimes_{\Coeff}N)$.
\end{enumerate}
\end{lem}
\begin{proof}
(1) follows from an argument of \cite[Proposition~2.6]{MR3230821}.
For (2), we may assume $N = \pi_{I,\emptyset,*}(N_0)$.
Then $\Hom(M,N)\simeq\Hom(\pi_{I,\emptyset}^*M,N_0)$ and $\Hom(\Coeff'\otimes_{\Coeff}M,\Coeff'\otimes_{\Coeff}N)\simeq \Hom(\Coeff'\otimes_{\Coeff}\pi_{I,\emptyset}^*M,\Coeff'\otimes_{\Coeff}N_0)$.
Hence we may assume $I = \emptyset$.
Moreover we may assume $M = B_{s_1}\otimes\dotsb\otimes B_{s_l}$ and $N = B_{t_1}\otimes\dotsb\otimes B_{t_r}$ for some $s_1,\dots,s_l,t_1,\dots,t_r\in S$.
Then $\Hom(M,N)$ and $\Hom(\Coeff'\otimes_{\Coeff}M,\Coeff'\otimes_{\Coeff}N)$ has a basis called double leaves~\cite[Theorem~5.5]{MR4321542} (cf.~\cite{MR2441994}).
From the construction of double leaves basis, they correspond to each other by the base change to $\Coeff'$.
Hence we have (2).
\end{proof}

\begin{lem}
The functor $\presubscript{I}{\mathbb{H}}\colon \Parity_B(\presubscript{I}{X})\to \Sbimod{I}$ is fully-faithful.
\end{lem}
\begin{proof}
Let $\mathcal{F},\mathcal{G}\in \Parity_{B}(\presubscript{I}{X})$.
We prove that
\begin{equation}\label{eq:map between Hom}
\Hom_{\Parity_B(\presubscript{I}{X})}(\mathcal{F},\mathcal{G})\to \Hom_{\Sbimod{I}}(\presubscript{I}{\mathbb{H}}(\mathcal{F}),\presubscript{I}{\mathbb{H}}(\mathcal{G}))
\end{equation}
is bijective.
By Lemma~\ref{lem:hom formula on Parity}, it is sufficient to prove that \eqref{eq:map between Hom} is surjective.
It is sufficient to prove that \eqref{eq:map between Hom} is surjective after localizing at any maximal ideal $\mathfrak{m}\subset \Coeff$ and by Nakayama's lemma, it is sufficient to prove the surjectivity after tensoring with $\Coeff/\mathfrak{m}$.
By the previous lemma, it is sufficient to prove that \eqref{eq:map between Hom} is surjective after replacing $\Coeff$ with $\Coeff/\mathfrak{m}$.
Namely we may assume that $\Coeff$ is a field.

So we assume that $\Coeff$ is a field in the rest of the proof.
By Lemma~\ref{lem:hom formula on Parity}, it is sufficient to prove that \eqref{eq:map between Hom} is injective.
Let $Y = \bigcup_{w\in A}\presubscript{I}{X}_{w}\subset\presubscript{I}{X}$ be a closed subset such that $\#A < \infty$ and $f\colon Y\hookrightarrow\presubscript{I}{X}$ the inclusion map.
We prove $\Hom(\mathcal{F},f_*f^!\mathcal{G})\to\Hom(\presubscript{I}{\mathbb{H}}(\mathcal{F}),\presubscript{I}{\mathbb{H}}(f_*f^!\mathcal{G}))$ is injective by induction on $\#A$.

Let $w\in A$ be a maximal element and set $Z = Y\setminus \presubscript{I}{X}_{w}$, $U = \presubscript{I}{X}\setminus Z$.
Let $i\colon Z\hookrightarrow \presubscript{I}{X}$ and $j\colon U\hookrightarrow\presubscript{I}{X}$ be the inclusion maps.
Set $A' = A\setminus\{w\}$.
Then we have the following commutative diagram
\begin{equation}\label{eq:commutative diagram, for fully-faithfulness}
\begin{tikzcd}
 & 0\arrow[d]\\
\Hom(\mathcal{F},i_*i^!\mathcal{G})\arrow[r]\arrow[d] & \Hom(\presubscript{I}{\mathbb{H}}(\mathcal{F}),\presubscript{I}{\mathbb{H}}(i_*i^!\mathcal{G}))\arrow[d]\\
\Hom(\mathcal{F},f_*f^!\mathcal{G})\arrow[r]\arrow[d] & \Hom(\presubscript{I}{\mathbb{H}}(\mathcal{F}),\presubscript{I}{\mathbb{H}}(f_*f^!\mathcal{G}))\arrow[d]\\
\Hom(\mathcal{F},j_{w*}j_{w}^!\mathcal{G})\arrow[r] & \Hom(\presubscript{I}{\mathbb{H}}(\mathcal{F}),\presubscript{I}{\mathbb{H}}(j_{w*}j_{w}^!\mathcal{G})).
\end{tikzcd}
\end{equation}
Here we use $i_*i^!f_*f^! \simeq i_*i^!$ and $j_*j^!f_*f^!\simeq j_{w*}j_w^!$.
By Lemma~\ref{lem:relation of restriction between sheaf and Soergel bimodule, locally closed}, the right column is exact.
The first row is injective by inductive hypothesis.
It is sufficient to prove that the last row is injective.

Note that we have an equality
\[
\Hom_{H^{\bullet}_B(\presubscript{I}{X}_w)}(H_B^{\bullet}(\presubscript{I}{X}_w,j_{w}^*\mathcal{F}),H^{\bullet}_B(\presubscript{I}{X}_w,j_{w}^!\mathcal{G})) = \Hom_{\BigCat{I}}(\presubscript{I}{\mathbb{H}}(j_{w*}j_{w}^*\mathcal{F}),\presubscript{I}{\mathbb{H}}(j_{w*}j_{w}^!\mathcal{G})).
\]
Indeed, we have $H_B^{\bullet}(\presubscript{I}{X}_w,j_{w}^*\mathcal{F}) = \presubscript{I}{\mathbb{H}}(j_{w*}j_{w}^*\mathcal{F})$, $H_B^{\bullet}(\presubscript{I}{X}_w,j_{w}^!\mathcal{G}) = \presubscript{I}{\mathbb{H}}(j_{w*}j_{w}^!\mathcal{G})$ and $H_B^{\bullet}(\presubscript{I}{X}_w)\simeq R$.
Hence the right hand side is contained in the left hand side.
We have $\supp_W(\presubscript{I}{\mathbb{H}}(j_{w*}j_{w}^*\mathcal{F}))\subset \{w\}$ as in the proof of Lemma~\ref{lem:relation of restriction between sheaf and Soergel bimodule}.
Since $H_B^{\bullet}(\presubscript{I}{X}_w,j_{w}^*\mathcal{F}) = \presubscript{I}{\mathbb{H}}(j_{w*}j_{w}^*\mathcal{F})$ is a free $R$-module, we have $\presubscript{I}{\mathbb{H}}(j_{w*}j_{w}^*\mathcal{F})\hookrightarrow \presubscript{I}{\mathbb{H}}(j_{w*}j_{w}^*\mathcal{F})_Q^w$ and the same is true for $j_{w*}j_{w}^!\mathcal{G}$.
Therefore any $R$-module homomorphism $\presubscript{I}{\mathbb{H}}(j_{w*}j_{w}^*\mathcal{F})\to \presubscript{I}{\mathbb{H}}(j_{w*}j_{w}^!\mathcal{G})$ is a morphism in $\BigCat{I}$, hence we have the above equality.

The last row of \eqref{eq:commutative diagram, for fully-faithfulness} is decomposed into
\begin{align*}
\Hom_{D_{B}^{\mathrm{b}}(\presubscript{I}{X})}(\mathcal{F},j_{w*}j_{w}^!\mathcal{G})
& \simeq\Hom_{D_{B}^{\mathrm{b}}(\presubscript{I}{X}_w)}(j_{w}^*\mathcal{F},j_{w}^!\mathcal{G})\\
& \to \Hom_{H^{\bullet}_B(\presubscript{I}{X}_w)}(H_B^{\bullet}(\presubscript{I}{X}_w,j_{w}^*\mathcal{F}),H^{\bullet}_B(\presubscript{I}{X}_w,j_{w}^!\mathcal{G}))\\
& = \Hom_{\BigCat{I}}(\presubscript{I}{\mathbb{H}}(j_{w*}j_{w}^*\mathcal{F}),\presubscript{I}{\mathbb{H}}(j_{w*}j_{w}^!\mathcal{G}))\\
& \to \Hom_{\BigCat{I}}(\presubscript{I}{\mathbb{H}}(\mathcal{F}),\presubscript{I}{\mathbb{H}}(j_{w*}j_{w}^!\mathcal{G})),
\end{align*}
here the last map is induced by $\mathcal{F}\to j_{w*}j_{w}^*\mathcal{F}$.
The second morphism is an isomorphism since $j_w^*\mathcal{F}$ and $j_{w}^!\mathcal{G}$ are constant.
Therefore it is sufficient to prove that the last map is injective.

We have the commutative diagram
\[
\begin{tikzcd}
\Hom(\presubscript{I}{\mathbb{H}}(\mathcal{F}),\presubscript{I}{\mathbb{H}}(j_{w*}j_{w}^!\mathcal{G}))\arrow[r] & \Hom(\presubscript{I}{\mathbb{H}}(j_{!}j^*\mathcal{F}),\presubscript{I}{\mathbb{H}}(j_{w*}j_{w}^!\mathcal{G}))\\
\Hom(\presubscript{I}{\mathbb{H}}(j_{w*}j_{w}^*\mathcal{F}),\presubscript{I}{\mathbb{H}}(j_{w*}j_{w}^!\mathcal{G}))\arrow[u]\arrow[r] & \Hom(\presubscript{I}{\mathbb{H}}(j_{w!}j_{w}^*\mathcal{F}),\presubscript{I}{\mathbb{H}}(j_{w*}j_{w}^!\mathcal{G}))\arrow[u].
\end{tikzcd}
\]

The right column is injective by Lemma~\ref{lem:relation of restriction between sheaf and Soergel bimodule, locally closed}.
Therefore it is sufficient to prove the lower row is injective.

Note that $H_{B}^{\bullet}(\presubscript{I}{X}_w)\simeq H_B^{\bullet}(\mathrm{pt})\simeq R$.
Since $j_w^*\mathcal{F}$ is constant and $\presubscript{I}{X_w}\simeq \mathbb{A}^{\ell(w_-)}$, the natural map $H_{B}^{\bullet}(\presubscript{I}{X}_w,j_w^*\mathcal{F})\otimes_{H_{B}^{\bullet}(\presubscript{I}{X}_w)} H_{B,c}^{\bullet}(\presubscript{I}{X_w})\to H_{B,c}^{\bullet}(\presubscript{I}{X}_w,j_w^*\mathcal{F})$ is an isomorphism and $H_{B,c}^{\bullet}(\presubscript{I}{X_w})$ is free of rank one as an $H_{B}^{\bullet}(\presubscript{I}{X_w})\simeq R$-module generated by some $u\in H^{2\ell(w_-)}_{B,c}(\presubscript{I}{X_{w}})$.
Let $a\in H_B^{2\ell(w_-)}(\presubscript{I}{X_w})\subset H^{\bullet}_{B}(\presubscript{I}{X}_{w})\simeq R$ be the image of $u$ under $H_{B,c}^{\bullet}(\presubscript{I}{X}_w)\hookrightarrow H_{B}^{\bullet}(\presubscript{I}{X}_w)$.
Then we have $\presubscript{I}{\mathbb{H}}(j_{w!}j_{w}^*\mathcal{F}) = H_{B,c}^{\bullet}(\presubscript{I}{X}_w,j_{w}^*\mathcal{F}) \simeq H_B^{\bullet}(\presubscript{I}{X}_w,j_{w}^*\mathcal{F})a = \presubscript{I}{\mathbb{H}}(j_{w*}j_{w}^*\mathcal{F})a$.
The $R$-module $\presubscript{I}{\mathbb{H}}(j_{w*}j_w^!\mathcal{G})$ is free, hence it is torsion-free.
Hence the map $\Hom(\presubscript{I}{\mathbb{H}}(j_{w*}j_{w}^*\mathcal{F}),\presubscript{I}{\mathbb{H}}(j_{w*}j_{w}^!\mathcal{G}))\to \Hom(\presubscript{I}{\mathbb{H}}(j_{w!}j_{w}^*\mathcal{F}),\presubscript{I}{\mathbb{H}}(j_{w*}j_{w}^!\mathcal{G}))$ which is given by the restriction to $\presubscript{I}{\mathbb{H}}(j_{w*}j_w^*\mathcal{F})a$, is injective.
\end{proof}

\begin{proof}[{Proof of Theorem~\ref{thm:equivalence}}]
We prove that $\presubscript{I}{\mathbb{H}}$ is essentially surjective.
Let $M\in \Sbimod{I}$.
Then $M$ is a direct summand of a direct sum of objects of a form $\pi_{I,\emptyset,*}(B_{s_1}\otimes\dotsb \otimes B_{s_l})(n)$ for some $s_1,\dots,s_l\in S$ and $n\in\Z$.
By Lemma~\ref{lem:monoidal structure and H} and Proposition~\ref{prop:H and pi_*}, there exists $\mathcal{F}\in \Parity_B(\presubscript{I}{X})$ such that $M$ is a direct summand of $\presubscript{I}{\mathbb{H}}(\mathcal{F})$.
Since $\presubscript{I}{\mathbb{H}}$ is fully-faithful, there exists $\mathcal{F}'$ such that $M\simeq \presubscript{I}{\mathbb{H}}(\mathcal{F}')$.

Let $\presubscript{I}{\mathbb{H}}^{-1}$ be the quasi-inverse functor of $\presubscript{I}{\mathbb{H}}$.
Then we have $\presubscript{I}{\mathbb{H}}^{-1}\circ\pi_{I,J,*}\simeq \pi_{I,J,*}\circ\presubscript{I}{\mathbb{H}}^{-1}$ by Proposition~\ref{prop:H and pi_*}.
By taking the left adjoint functors of both sides, we get $\pi_{I,J}^*\circ\presubscript{I}{\mathbb{H}}\simeq \presubscript{J}{\mathbb{H}}\circ\pi_{I,J}^*$.
\end{proof}

\end{document}